\documentclass[11pt]{amsart}
\usepackage{amsmath,amsthm,amssymb,latexsym,graphicx,verbatim,enumerate,
ifpdf,mdwlist,xspace}
\usepackage{mathabx}
\ifpdf
\usepackage{hyperref}
\else
\usepackage[hypertex]{hyperref}
\fi
\usepackage{color}
\usepackage{amssymb, amsmath, amsthm}
\usepackage{graphicx}
\usepackage{cancel}
\usepackage{float}

\providecommand{\set}[1]{\lbrace #1 \rbrace}

\providecommand{\dom}{{\rm{dom}}}

\DeclareMathOperator{\orb}{orb}
\DeclareMathOperator{\Id}{Id}

\DeclareMathOperator{\col}{coll}
\DeclareMathOperator{\card}{card}

\newtheorem{question}{Question}
\newtheorem{thm}{Theorem}[section]
\newtheorem{definition}[thm]{Definition}

\newtheorem{remark}[thm]{Remark}

\newtheorem{lemma}[thm]{Lemma}

\newtheorem{prop}[thm]{Proposition}

\newtheorem{fact}[thm]{Fact}
\newenvironment{defn}{\begin{definition} \rm}{ \end{definition}}

\renewcommand{\phi}{\varphi}

\DeclareMathOperator{\range}{range}

\newcommand{\rel}[1]{\mathrel{#1}}

\title[Equivalence relations in the Ershov hierarchy]{Classifying equivalence
relations\\ in the Ershov hierarchy}

\author[N.~Bazhenov]{Nikolay Bazhenov}
\address{Sobolev Institute of Mathematics, pr. Akad. Koptyuga 4, Novosibirsk,
630090 Russia;\
Novosibirsk State University, ul. Pirogova 2, Novosibirsk, 630090 Russia
}
\email{bazhenov@math.nsc.ru}

\author[M.~Mustafa]{Manat Mustafa}
\address{Department of Mathematics,
School of Science and Technology, Nazarbayev University,
53, Kabanbay Batyr Avenue,
Astana, 010000, Republic of Kazakhstan}
\email{manat.mustafa@nu.edu.kz}

\author[L.~San Mauro]{Luca San Mauro}
\address{Institute of Discrete Mathematics and Geometry, Vienna University
of Technology, Vienna, Austria}
\email{luca.san.mauro@tuwien.ac.at}

\author[A.~Sorbi]{Andrea Sorbi}
\address{Dipartimento di Ingegneria Informatica e Scienze Matematiche\\
Universit\`a Degli Studi di Siena\\
I-53100 Siena, Italy}
\email{andrea.sorbi@unisi.it}

\author[M.~Yamaleev]{Mars Yamaleev}
\address{Kazan Federal University, 18 Kremlyovskaya str., Kazan, 420008,
Russia}
\email{mars.yamaleev@kpfu.ru}
\urladdr{}

\thanks{
Bazhenov, Mustafa, and Yamaleev were supported by Nazarbayev University Faculty
Development Competitive Research Grants N090118FD5342.
San Mauro was supported by the Austrian Science Fund FWF, project~M~2461.
 Sorbi is a member of INDAM-GNSAGA and his research was supported by
the PSR program of the University of Siena}


\begin{document}

\maketitle

\begin{abstract}
Computably enumerable equivalence relations (ceers) received a lot of
attention in the literature. The standard tool to classify ceers  is
provided by the computable reducibility $\leq_c$. This gives rise to a rich
degree-structure. In this paper, we lift the study of $c$-degrees to the
$\Delta^0_2$ case. In doing so, we rely on the Ershov hierarchy. For any
notation $a$ for a non-zero computable ordinal, we prove several algebraic
properties of the degree-structure induced by $\leq_c$ on the
$\Sigma^{-1}_{a}\smallsetminus \Pi^{-1}_a$ equivalence relations. A special
focus of our work is on the (non)existence of infima and suprema of
$c$-degrees.
\end{abstract}

\section{Introduction}
Computable reducibility is a longstanding notion that allows classifying
equivalence relations on natural numbers   according to  their complexity.

\begin{defn}
Let $R,S$ be equivalence relations with domain $\omega$. $R$ is \emph{computably
reducible} to $S$, denoted $R \leq_c S$, if there is a total computable
function $f$ such that, for all $x,y \in \omega$,
\[
xRy \Leftrightarrow f(x)\rel{S}f(y).
\]
\end{defn}
We write $f: R \leq_c S$ to denote that $f$ is a computable function that
reduces $R$ to $S$; $c$-degrees are introduced in the standard way.

\smallskip

The history of computable reducibility has many roots, being often
rediscovered and explored in connection with different fields. Its study
dates back to the fundamental  work of  Ershov in the theory of numberings,
where the reducibility is introduced in a category-theoretic fashion (see
Ershov's monograph~\cite{Ershov:77} in Russian, or \cite{Ershov:survey} for
an English survey). In the 1980s, computable reducibility proved to be a
fruitful tool for calibrating the complexity of provable equivalence of
formal systems  and scholars focused mostly on the $\Sigma^0_1$ case (see,
e.g., \cite{Visser:80,Montagna-Sorbi:Universal,Bernardi:83}). Following Gao
and Gerdes~\cite{Gao:01}, we adopt the acronym ``ceers'' to refer to
computably enumerable equivalence relations.  The interested reader can
consult Andrews, Badaev, and Sorbi~\cite{andrews2017survey} for a nice and
up-to-date survey on ceers, with a special focus on \emph{universal ceers},
i.e., ceers to which all other ceers are computably reducible. The degree of
universal ceers is by now significantly explored: for instance, in
\cite{Andrews:14} the authors proved that all uniformly effectively
inseparable ceers are universal. A complementary line of research aims at
providing concrete examples of universal ceers. To this end, Nies and Sorbi
\cite{nies2018calibrating} constructed a finitely presented group whose word
problem  is a universal ceer.

Far from being limited to ceers, computable reducibility has been also
applied to equivalence relations of much higher complexity. Fokina, Friedman,
Harizanov, Knight, McCoy, and Montalb\'an~\cite{Fokina:12b} showed that all
$\Sigma^1_1$ equivalence relations are computably reducible to the
isomorphism relations on several classes of computable structures (e.g.,
graphs, trees, torsion abelian groups, fields of characteristic 0 or $p$,
linear orderings). This study was fueled by the observation that computable
reducibility  represents a nice effective counterpart of \emph{Borel
reducibility}, i.e., a key notion of modern descriptive set theory (see
~\cite{friedman1989borel}). The analogy between Borel and computable
reducibility has been explored, for instance, by Coskey, Hamkins, and
Miller~\cite{Coskey:12}, that investigated equivalence relations on c.e. sets
mirroring classical combinatorial equivalence relations of fundamental
importance for Borel theory.

Additional motivation for dealing with computable reducibility comes from the
study  of c.e.\ presentations of structures, as is shown for instance in
\cite{fokina2016linear,gavruskin2014graphs} (for a nice survey about c.e.\
structures, see \cite{Selivanov03}).

\smallskip

The goal of the present paper is to contribute to this vast (yet somehow
unsystematic) research program by making use of computable reducibility to
initiate a throughout classification of the complexity of $\Delta^0_2$
equivalence relations. In this endeavour, we follow and extend the work of
Andrews and Sorbi~\cite{andrewssorbi2}, that provides a very extensive
analysis of the degree structure induced by computable reducibility on ceers.
Ng and Yu~\cite{ngyu} broadened the perspective by discussing some structural
aspects of the $c$-degrees of $n$-c.e., $\omega$-c.e., and
$\Pi^0_1$-equivalence relations. We similarly rely on the Ershov hierarchy to
pursue our analysis.

Although our motivation is rather abstract (and to some extent corresponds to
the desire of exporting the guiding questions of classical degree theory to
the case of equivalence relations), our object of study shall not be regarded
as too much artificial. The following example might convince the reader that
 $\Delta^0_2$ equivalence relations occur quite naturally.

Consider the following $\Pi^0_2$ equivalence relation $R$:
\[
	i \rel{R} j \ \Leftrightarrow\ \card(W_i) = \card(W_j).
\]
Then one can define a ``bounded'' version of $R$:
\[
\langle i,s\rangle \rel{R^b} \langle j,t\rangle \text{~if~and~only~if~}
\card(W_i \cap \{ 0,1,\dots,s\}) = \card (W_j \cap \{ 0,1,\dots,t\}).
\]
It is not hard to show that the relation $R^b$ is $\omega$-c.e. Furthermore,
the relation $R^b$ admits an interpretation via algebraic structures: One can
interpret a number $\langle i,s\rangle$ as the index of a finite linear
ordering. Indeed, define the ordering $\mathcal{L}_{ i,s}$ as follows. The
domain of $\mathcal{L}_{i,s}$ is equal to $W_{i}\cap \{ 0,1,\dots,s\}$, and
the ordering on the domain is induced by the standard ordering of natural
numbers. Notice that here we assume that $\mathcal{L}_{i,s}$ may be empty.
The list $(\mathcal{L}_{i,s})_{i,s\in\omega}$ gives an enumeration of all
finite linear orderings, up to permutations of the domains. It is easy to see
that
\[
\langle i,s\rangle \rel{R^b} \langle j,t\rangle \ \Leftrightarrow
\mathcal{L}_{i,s} \cong \mathcal{L}_{j,t},
\]
thus the relation $R^b$ can be treated as (one of the possible formalizations
of) the relation of isomorphism on the class of finite linear orderings.

\subsection{Organization of the paper}
In Section $2$, we set up the stage by offering some disanalogies between the
degree-structure of ceers and that of $\Delta^0_2$ equivalence relations. We
also prove that infinitely many levels of the Ershov hierarchy contain
minimal $c$-degrees. In Section $3$, we focus on dark degrees, i.e.,
$c$-degrees  not being above the identity on $\omega$: we show that all
levels of the Ershov hierarchy, with the exception of $\Pi^{-1}_1$, contain
dark equivalence relations. Section $4$ and Section $5$ are devoted to the
existence of infima and suprema of $c$-degrees of $\Delta^0_2$-equivalence
relations: we introduce the notion of mutual darkness and prove that, if $R,S
\in \Sigma^{-1}_a\smallsetminus \Pi^{-1}_a$ are mutually dark, then $R,S$
have no infimum in $\Sigma^{-1}_a\smallsetminus \Pi^{-1}_a$ and no supremum
in $\Delta^0_2$. It follows that none of the degree-structures considered in
this paper is neither upper- or a lower-semilattice.

\subsection{Notation and terminology}
All our equivalence relations have domain $\omega$. Given  a number $n$, we
denote by $[x]_R$ its $R$-equivalence class. We say that $R$ is
\emph{infinite} if $R$ has infinitely many equivalence classes (otherwise, it
is of course \emph{finite}). The following basic equivalence relations will
appear many times:

\begin{itemize}
\item $\Id_n$ is the computable equivalence relation consisting of $n$
    equivalence classes, i.e.
\[
x \rel{\Id_n} y \Leftrightarrow x\equiv y \mbox{ (mod }  n),
\]
for all $x,y\in \omega$.
\item $\Id$ is the identity on $\omega$, i.e.,
    $x \rel{\Id} y$ if and only if $x=y$.
\end{itemize}

The following definition is due to Gao and Gerdes~\cite{Gao:01} (but
analogous ways of coding sets of numbers by equivalence relations occur
frequently in the literature, see for instance the definition of a
\emph{set-induced} $c$-degree in \cite{ngyu}).

\begin{defn}
An equivalence relation $R$ is \emph{$n$-dimensional} if there are pairwise
disjoint sets $A_0,\ldots,A_{n-1}\subseteq \omega$ such that
\[
x Ry \Leftrightarrow x=y \lor (\exists i)(x,y \in A_i).
\]
We denote such $R$ by $R_{A_0,\ldots,A_{n-1}}$.

An equivalence relation $R$ is \emph{essentially $n$-dimensional} if it has
exactly $n$ noncomputable equivalence class.
\end{defn}

The next definition appears in \cite{andrewssorbi2} and  reflects a
fundamental distinction in how much information one can effectively extract
from a given equivalence relation $R$.

\begin{defn}\label{defn:light}
An equivalence relation $R$ is \emph{light} if $\Id\leq_c R$; it is
\emph{dark} if it is not light and has infinitely many equivalence classes.
\end{defn}

It is often convenient to think of a given light equivalence relation $R$ in
terms of some computable listing of pairwise nonequivalent numbers witnessing
its lightness. More formally, a \emph{transversal} of an equivalence relation
$R$ with infinitely many $R$-classes is an infinite set $A$ such that
$[x]_R\neq [y]_R$, for all distinct $x,y \in A$. It is immediate to see that
$R$ is light if and only it has a c.e.\ transversal.

\smallskip

Our computability theoretic notions are standard,  see for
instance~\cite{Soare:Book}. The basic notions regarding the Ershov hierarchy
can be found in \cite{ershov-i, ershov-ii, ershov-iii}, see also
\cite{ash2000computable}: in particular, recall the following.

\begin{defn}\label{defn:Ershov}
Let $a$ be a notation for a computable ordinal. A set $A \subseteq \omega$ of
numbers is said to be $\Sigma^{-1}_a$ (or $A \in \Sigma^{-1}_a$) if there are
computable functions $f(z,t)$ and $\gamma(z,t)$ such that, for all $z$,
\begin{enumerate}
  \item\label{1} $A(z)=\lim_t f(z, t)$, with $f(z,0)=0$;
  \item\label{2} $\gamma(z,0) =a$, and
       \begin{enumerate}
           \item $\gamma(z,t+1)\le_\mathcal{O}
               \gamma(z,t)\leq_\mathcal{O} a$;
           \item $f(z,t+1) \ne f(z,t) \Rightarrow \gamma(z,t+1)
               \ne \gamma(z,t)$.
        \end{enumerate}
  \item\label{3} for every $s$ there is at most one $x$ such that $f(x,s+1)
      \ne f(x,s)$.
\end{enumerate}
Item~\ref{3}. usually is not required in the literature, but it is clear
that it can be safely assumed without loss of generality: it implies that for
every $s$ we have $f(x,s)=0$ for cofinitely many $x$. We call the partial
function $\gamma$ the \emph{mind--change function} for $A$, relatively to
$f$.
\end{defn}

A \emph{$\Sigma^{-1}_a$--approximating pair} to a $\Sigma^{-1}_a$--set $A$,
is a pair $\langle f, \gamma\rangle$, where $f$ and $\gamma$ are computable
functions satisfying \ref{1}., \ref{2}., and \ref{3}. above, for $A$. As is
known (see, e.g., \cite{Seliv1985,ospichev2015friedberg} for more details),
one can effectively list all approximations $\langle f_e, \gamma_e \rangle$
to $\Sigma^{-1}_a$-sets. Thus we can refer to the listing $(E_e)_{e \in
\omega}$ of the $\Sigma^{-1}_a$-sets, where $E_e$ is the set of which
$\langle f_e, \gamma_e \rangle$ is a $\Sigma^{-1}_a$--approximating pair.

Dually, we say that a set $A$ is $\Pi^{-1}_a$, if $\overline{A} \in
\Sigma^{-1}_a$, or equivalently there is a \emph{$\Pi^{-1}_a$-approximating
pair} $\langle f, \gamma\rangle$, i.e. a pair  as above but starting with
$f(z,0)=1$. If $\mathcal{X} \in \{\Sigma^{-1}_a: a \in \mathcal{O}\}\cup
\{\Pi^{-1}_a: a \in \mathcal{O}\}$ let us call $\mathcal{X}^d=\{\overline{A}:
A \in \mathcal{X}\}$ the \emph{dual class} of $\mathcal{X}$.

We say that a set $A$ is \emph{properly $\mathcal{X}$} if $A \in \mathcal{X}
\smallsetminus \mathcal{X}^d$.

Since any finite ordinal has only one notation, one usually writes
$\Sigma^{-1}_n$  instead of $\Sigma^{-1}_a$, if $a$ is the notation of $n \in
\omega$. In analogy with the terminology used for sets, we say that $R$ is a
$n$-ceer if $R\in \Sigma^{-1}_n$.

\section{A first comparison with ceers}
At first sight, one might expect that the structural properties of the
$c$-degrees of ceers are reflected smoothly on the upper levels of the Ershov
hierarchy. In this section we show that the parallel is much more delicate.

\subsection{Equivalence relations with finitely many classes}
Recall that any ceer with finitely many classes is computable. Surely, this
is not the case for relations in the Ershov hierarchy:

\begin{lemma}\label{lemma:F_X}
For every non-zero natural number $n$, there is a $\Pi^{-1}_{2n}$ equivalence
relation $R$ such that $R$ is noncomputable and has finitely many equivalence
classes.
\end{lemma}
\begin{proof}
Consider a set $X\in \Sigma^{-1}_n\smallsetminus \Pi^{-1}_n$ and define the
relation
\begin{equation} \label{equ:F_X}
		F_X := \{ (x,y) \,\colon x,y\in X\ \text{or}\ x,y \in \overline{X}\}.
\end{equation}
It is straightforward to check that $F_X$ is a noncomputable equivalence
relation with two equivalence classes and is in $\Pi^{-1}_{2n}$.
\end{proof}

Relations of the form~(\ref{equ:F_X}) already allow us to demonstrate some
simple differences concerning elementary theories:

\begin{prop}
Suppose that $\mathcal{X} \in \{ \Pi^{-1}_a \,\colon |a|_{\mathcal{O}} \geq
2\} \cup \{ \Sigma^{-1}_{a}\,\colon |a|_{\mathcal{O}}\geq 3\}$. Then the
structure of $\mathcal{X}$-equivalence relations is elementarily equivalent
to neither ceers nor co-ceers.
\end{prop}
\begin{proof}
Fix a c.e.\ set $W$ such that $\deg_m(W)$ is a minimal $m$-degree. Note that
$\Id_2 <_c F_W$. Suppose that $E$ is an equivalence relation such that
$\Id_2\leq_c E\leq_c F_W$. Then $E$ is equal to $F_{V}$ for some c.e. set $V$
and the minimality of $\deg_m(W)$ implies that either $V$ is computable or
$V\equiv_m W$. Thus, we have $E\equiv_c \Id_2$ or $E\equiv_c F_{W}$. 	

Hence, the desired elementary difference can be witnessed by the following
argument:
\begin{enumerate}	
    \item The ($c$-degree of the) relation $\Id_1$ is the least element
        under computable reducibility (in ceers, co-ceers, and
        $\mathcal{X}$).
		
	\item $\Id_2$ is the unique minimal $c$-degree over $\Id_1$ (in ceers,
co-ceers, and $\mathcal{X}$).
		
    \item Inside $\mathcal{X}$, one can find two incomparable elements
        $x_0$ and $x_1$ (namely, the $c$-degrees of $\Id_3$ and $F_W$) such
        that $\Id_2 < x_i$ and $\neg \exists z ( \Id_2 < z < x_i )$. Note
        that this property fails for ceers and co-ceers.
\end{enumerate}
\end{proof}

\subsection{Finite minimality}
We move now to equivalence relations with infinitely many equivalence classes.

Notice that in the context of ceers, explored in \cite{andrewssorbi2},
equivalence relations with finitely many classes are computable, whereas this
is not so in higher levels of the Ershov hierarchy as witnessed by the
two-classes equivalence relations of the form $F_X$ introduced in the proof
of Lemma~\ref{lemma:F_X}. This suggests the following notion of minimality
(called finite minimality) when we work in the Ershov hierarchy.

\begin{defn}
An equivalence relation $R$ is \emph{finitely minimal} if $S<_c R$ implies that
$S$ has only finitely many equivalence classes.
\end{defn}

Observe also that if $R$ is light and $R \not\leq_c \Id$, then $R$ is not
finitely minimal.

The previous definition does perfect justice to the notion of minimality for
dark equivalence relations, as it is easy to see (see \cite{andrewssorbi2}
where the property is stated for ceers, but it clearly holds of all
equivalence relations) that if $E$ is dark, $R \leq_c E$ and $R$ has
infinitely many equivalence classes, then $R$ is dark as well. So a finitely
minimal dark equivalence relation is exactly a minimal dark equivalence
relation, i.e. a dark equivalence relation for which there is no dark
equivalence relation strictly below it.

One of the tools that will be useful to us is the \emph{collapse technique},
extensively used for ceers in \cite{andrewssorbi2}. If $R$ is an equivalence
relation and $x \cancel{\rel{R}} y$, then the \emph{collapse} $R_{\col(x,y)}$
(denoted by $R_{/(x,y)}$ in \cite{andrewssorbi2}) is defined as
\[
R_{\col(x,y)} := R \cup \{ (u,v) \,\colon u,v \in [x]_R \cup [y]_R\}.
\]
We illustrate the technique by obtaining the next two results about finite
minimality.

\begin{lemma}\label{lemma:finite-minimal}
Suppose that $R$ is dark with a computable class. Then $R$ cannot be finitely
minimal.
\end{lemma}

\begin{proof}
Suppose that classes $[a]_R$ and $[b]_R$ are distinct, and $[b]_R$ is
computable. Consider the collapse $S := R_{\col(a,b)}$. Then (by essentially
the same argument as in the proof of \cite[Lemma~2.6]{andrewssorbi2}) $S$ is
reducible to $R$ by the function
\[
		f(x) = \left\{
			\begin{array}{ll}
				x, & \text{if~} x\not\in [b]_R,\\
				a, & \text{otherwise}.
			\end{array}
		\right.
\]
Assume that $g\colon R\leq_c S$. Consider the map $h := f\circ g$ and
the \emph{$h$-orbit} of $b$, i.e. the set
\[
\orb_h(b) = \{ h^k(b) \,\colon k\in\omega\}.
\]
It is easy to see that $h\colon R\leq_c R$. We claim that the orbit $\orb_h(b)$
consists of pairwise non-$R$-equivalent elements: indeed, if $h^k(b) \rel{R}
h^l(b)$ for  $k<l$, then we have $b \rel{R} h^{l-k}(b)$, and
$\range(f) \cap [b]_R\neq \emptyset$, which contradicts the choice of the map
$f$. Hence, $\orb_h(b)$ is an infinite c.e. transversal of $R$, and we obtain
a contradiction with the darkness of $R$. Therefore, we deduce that $S<_c R$
and $R$ is not finitely minimal.
\end{proof}

Now we show that there are infinitely many levels of the Ershov hierarchy
which contain finitely minimal, dark equivalence relations \emph{properly}
belonging to the level:

\begin{thm}\label{theo:dark-minimal}
Suppose that $a\in \mathcal{O}$ and $|a|_{\mathcal{O}} \geq 1$. Then there
exists a finitely minimal, dark equivalence relation  $R \in \Delta^0_2
\smallsetminus \Sigma^{-1}_a$.
\end{thm}

\begin{proof}
In the proof of \cite[Thm. 3.3]{andrewssorbi2}, Andrews and Sorbi constructed
infinitely many pairwise incomparable, finitely minimal, dark ceers $S_l$,
$l\in \omega$, with the following property: for any $e$ and $l$, if the c.e.
set $W_e$ intersects infinitely many $S_l$-classes, then it intersects
\emph{every} $S_l$-class. We choose only one such a ceer $S := S_0$. 	As
explained in the introduction, fix a $\Sigma^{-1}_a$ list
$(E_e)_{e\in\omega}$ of all $\Sigma^{-1}_a$ sets
$(E_e)_{e\in\omega}$. We build a $\Delta^0_2$ equivalence relation $R$ with
the following properties:
\begin{itemize}
\item $R\supseteq S$, and
		
\item $R\neq E_e$, for every $e\in\omega$.
\end{itemize}
	The construction proceeds in a straightforward $\mathbf{0}'$-effective
manner. We choose a $\mathbf{0}'$-effective list $\{ x_e\}_{e\in\omega}$
which enumerates representatives of all $S$-clas\-ses, without repetitions
(i.e. $x_i \cancel{S} x_j$ for $i\neq j$). We start with $R[0]=S$, i.e. $S$
is the approximation $R[0]$ of $R$ at stage $0$.

Consider stage $k$. If $x_{2k} \cancel{E_k} x_{2k+1}$, then $R[k+1]$ is equal
to the collapse $R[k]_{\col(x_{2k},x_{2k+1})}$. Otherwise, $R[k+1] = R[k]$.

As per usual, set $R = \bigcup_{k\in\omega} R[k]$. First notice that $R$ is
$\Delta^0_2$: to see if $x \rel{R} y$ use oracle $\mathbf{0}'$ to search for
the unique $h,k$ such that $x \rel{S} x_h$ and $y \rel{S} x_k$, and then
check if $x_h=x_k$ or $x_h$ and $x_k$ have been collapsed at stage $\max(h,k)$.
Finally, it is not hard to show that $R$ has infinitely many classes, and
$R\not\in \Sigma^{-1}_a$.

Assume now that $f\colon Id\leq_c R$. Since $R\supseteq S$, the map $f$ is
also a reduction from $Id$ to $S$, which contradicts the darkness of $S$.
Thus, $R$ is also dark.

Suppose that $g\colon Q\leq_c R$ and $Q$ has infinitely many classes. Since
$Q$ contains infinitely many classes, the set $W := \range(g)$ is a c.e. set
which intersects infinitely many $R$-classes. Recall that $R\supseteq S$,
hence, $W$ intersects infinitely many $S$-classes. The choice of the ceer $S$
implies that $W$ intersects every $S$-class.

Now we build a map $h$ as follows: Fix an effective approximation $\{
S[t]\}_{t\in\omega}$ of the ceer $S$ and for a number $x$, define $h(x)$ to
be the first seen $y$ such that $g(y) \rel{S} x$: more formally,
\begin{gather*}
	t(x) := \mu t [ (\exists y\leq t) ( (x, g(y)) \in S[t] ) ],\\
	h(x) := \mu y [(x, g(y)) \in S[t(x)]].
\end{gather*}
Since $\range(g)$ intersects every $S$-class, $h$ is a total computable
function. We show that $h\colon R\leq_c Q$. Note that for any $x,y\in
\omega$, we have $x \rel{S} gh(x)$, and the following conditions are
equivalent:
\[
	h(x) \rel{Q} h(y)  \Leftrightarrow g(h(x)) \rel{R} g(h(y))
\Leftrightarrow\ x \rel{R} y.
\]
Thus, we have $Q\equiv_c R$, and $R$ is finitely minimal.
Theorem~\ref{theo:dark-minimal} is proved.
\end{proof}

Note that in the relation $R$ from the theorem above, every $R$-class is a
c.e. set. Thus, $R$ has the following curious property: If $Q\leq_c R$ and
$Q$ has only finitely many classes, then $Q\equiv_c Id_n$ for some
$n\in\omega$.

\subsection{Failure of Inversion Lemma}
A fundamental technique when studying ceers is provided by the following
result: see for instance \cite[Lemma~1.1]{andrewssorbi2}.

\begin{lemma}[Inversion Lemma]
Suppose that $R,S$ are ceers and $R\leq_c S$ via $f$. If $f$ hits all the
equivalence classes of $S$ (i.e. $\range(f)$ intersects with \emph{every}
$S$-class), then $S\leq_c R$.
\end{lemma}

The Inversion Lemma does not hold, in general, for the Ershov hierarchy. In
fact, it fails already for $2$-ceers:

\begin{lemma}
There is an equivalence relation $R\in \Sigma^{-1}_{2} \smallsetminus
\Sigma^{-1}_{1}$ such that  $Id\leq_c R$ via a function $f$ which is
surjective on the equivalence classes of $R$.
\end{lemma}
\begin{proof}
Fix a noncomputable c.e. set $X$. Split $\omega$ into four computable parts:
$A = \{ a_0,a_1,a_2,\dots\}$, $B = \{ b_0, b_1, b_2, \dots\}$, $C = \{ c_0,
c_1, c_2,\dots\}$, and $D = \{ d_0,d_1,d_2,\dots\}$.

The relation $R$ is given by its equivalence classes: for every $i\in\omega$,
\begin{itemize}
\item If $i\not\in X$, then $R$ contains disjoint classes $\{ a_i, c_i\}$
    and $\{ b_i, d_i\}$.
	
\item If $i\in X$, then there are $R$-classes $\{ a_i, d_i\}$ and $\{ b_i,
    c_i\}$.
\end{itemize}
It is clear that $R$ is not a ceer, and the function
\[
	f(2i) = a_i, \quad f(2i+1) = b_i,
\]
gives a reduction from $Id$ to $R$, hitting all the $R$-classes. Furthermore,
$R$ is a $2$-ceer, since after separating two classes in the approximation of
$R$ (due to some $i$ being enumerated in $X$), we never merge them again.
\end{proof}

\begin{lemma}
For any dark $\Delta^0_2$ equivalence relation $R$, there are a dark
equivalence relation $R<_c S$ and a reduction $f: R\leq_c S$ such that $f$
hits all the $S$-classes.
\end{lemma}
\begin{proof}
Fix a dark ceer $Q$. Choose a $\mathbf{0}'$-effective list $\{ r_i
\}_{i\in\omega}$ of representatives of all $R$-classes, without repetitions.
Similarly, choose a $\mathbf{0}'$-effective list $\{ q_i\}_{i\in\omega}$ for
representatives of $Q$-classes.

The construction is given in a $\mathbf{0}'$-computable way. At stage $0$,
set $S[0] := R\oplus Q$. At stage $e+1$,  if $\varphi_e(2r_e)\!\downarrow \!
\cancel{\rel{R}} \varphi_e(2q_e+1)\!\downarrow$, then set $S[e+1] :=
S[e]_{\col(2r_e,2q_e+1)}$. Otherwise, define $S[e+1]:=S[e]_{\col(0,2q_e+1)}$.

It is not hard to show that the constructed $S$ is $\Delta^0_2$, and $S\nleq_c
R$. Moreover, the function $f\colon x\mapsto 2x$ gives a reduction of $R$ to
$S$, hitting all the $S$-classes.

Assume that $g\colon Id\leq_c S$. Then either $\range(g)$ contains infinitely
many even numbers and $R$ cannot be dark, or $\range(g)$ contains infinitely
many odd numbers and $Q$ cannot be dark. In any case, this leads to a
contradiction, thus, $S$ is dark.
\end{proof}

\section{Dark equivalence relations in the Ershov Hierarchy}
In this section, we show that the phenomenon of  darkness is rather
pervasive: with the exception of the co-ceers, dark degrees exist properly at
each level of the Ershov hierarchy. In the construction of these degrees we
also develop some strategies that will be helpful in further sections, when
the focus will be in on the existence of infima and suprema of $c$-degrees.

\begin{prop}
There are no dark co-ceers.
\end{prop}

\begin{proof}
Suppose that $R\in \Pi^0_1$. If $R$ has finitely many equivalence classes,
then it is trivially not dark. Hence, assume there exist infinitely many
$R$-classes. We prove that $R$ is light by inductively building the following
c.e.\ transversal of $R$: Let $x_0=0$ and let $x_{i+1}$ be any number $z$
such that, for all $j\leq i$, $(x_j,z)\notin R$. Such a $z$ must exist
(otherwise, there would be only finitely many $R$-classes) and, since $R$ is
co-c.e.\, will be found effectively.
\end{proof}

\begin{thm}\label{thm:properly-dark}
If $\mathcal{X} \in \{\Sigma^{-1}_a: a \in \mathcal{O}, |a|_\mathcal{O}\ge
1\} \cup \{\Pi^{-1}_a: a \in \mathcal{O}, |a|_\mathcal{O}>1\}$ then there is
a dark equivalence relation having only finite equivalence classes, and
properly lying in $\mathcal{X}$.
\end{thm}

\begin{remark}\label{rem:theorems}
{\rm Theorem \ref{thm:properly-dark} can be obtained as a corollary of
Theorem \ref{thm:mutually-dark}. For the sake of exposition, being the proof
of Theorem \ref{thm:mutually-dark} more complicated, we shall provide a proof
of Theorem \ref{thm:properly-dark} nonetheless. This might help the reader to
familiarize, in a simpler context, with the techniques required for Theorem
\ref{thm:mutually-dark}, and in fact exploited also in the proof of Theorem
\ref{thm:omegainf}.}
\end{remark}

\begin{proof}[Proof of Theorem \ref{thm:properly-dark}]
Let us start with the case $\mathcal{X}=\Sigma^{-1}_a$, with
$|a|_\mathcal{O}\ge 1$. We want to build an equivalence relation $R$
satisfying the following requirements, for every $e \in \omega$:
\begin{align*}
  &F_{e}: \text{$[e]_{R}$ is finite,}\\
  &P_e:  \text{$W_e$ is not an infinite transversal for $R$,}\\
  &Q_e: A \ne E_e,
\end{align*}
where  $(E_e)_{e \in \omega}$ is a listing of the $\Pi^{-1}_a$ sets
as explained in the introduction, with $\langle f_{e}, \gamma_{e}\rangle$ a
$\Pi^{-1}_a$-approximation to $E_{e}$.

Let us define the \emph{priority ordering} of the requirements as
\[
F_{0}<Q_0<P_0< \cdots < F_{e}< Q_e < P_e < \cdots.
\]

We build computable functions $f(x,s)$ and $\gamma(x,s)$ so that the pair
$\langle f, \gamma \rangle$ is a $\Sigma^{-1}_a$-approximating pair to a
$\Sigma^{-1}_a$-equivalence relation $R$ defined as
\[
x \mathrel{R} y \Leftrightarrow \lim_s f(\langle x, y \rangle,s)=1,
\]
satisfying the given requirements.

\subsubsection*{Strategies for the requirements and their interactions}

The strategy for $F_{e}$ requires that no lower priority requirement adds any
element to $[e]_{R}$, which  is eventually finite because only
higher priority requirements may contribute with their actions to add
elements to $[e]_{R}$ (we will see every requirement may act only
finitely many times).

For the $Q$-requirement $Q_{e}$ we appoint as witness a pair
\[
(x_{e}, y_{e}) \in 2\omega \times (2\omega+1).
\]
$Q_{e}$ sets the restraint that no lower priority requirement may modify the
equivalence class of either $x_{e}$ or $y_{e}$.

The reason we choose $x_e$ to be even and $y_e$ to be odd (but it could be
the opposite as well) is for the sake of $P$-requirements. The requirement
$P_e$ waits for $W_{e}$ to enumerate a pair of distinct numbers $u, v$, both
even, or both odd. When found, it simply $R$-collapses $u, v$ (if not already
collapsed): this ensures that $W_{e}$ is not a transversal for $R$. Notice
that if $W_e$ is infinite then by the Pigeon Hole Principle it either
contains infinitely many even numbers or infinitely many odd numbers. We will
see by Lemma~\ref{lem:P-satisfied} that $P_e$ will not be restrained by
$\Sigma^{-1}_{a}$-ness from $R$-collapsing two even numbers or two odd
numbers, since the construction will ensure that a necessary condition, at
any stage, for which we may have $f(\langle x,y\rangle)=0$ but already
$\gamma(\langle x,y\rangle)=1$ is that $x,y$ have different parity.

As $f_{e}(\langle x_{e}, y_{e} \rangle,0)=1$ (and $\gamma_{e}(\langle x_{e},
y_{e} \rangle,0)=a$), the construction starts up with having $(x_{e}, y_{e})$
not in $R$, i.e. $f(\langle x_{e}, y_{e} \rangle,0)=0$, and $\gamma(\langle
x_{e}, y_{e} \rangle,0)=a$. Every time we see $f_{e}(\langle x_{e}, y_{e}
\rangle,s+1)\ne f_e(\langle x_{e}, y_{e} \rangle,s)$ , we change accordingly
$f(\langle x_{e}, y_{e} \rangle,s+1)$ as to have
\[
f(\langle x_{e}, y_{e} \rangle, s+1)  \ne  f_{e}(\langle x_{e}, y_{e}
\rangle,s+1),
\]
and we define
\[
\gamma(\langle x_{e}, y_{e} \rangle,s+1):= \gamma_{e}(\langle x_{e},
y_{e} \rangle, s+1).
\]
In this way $Q_{e}$ is able to diagonalize against $E_{e}$ at the witness
$(x_{e}, y_{e})$, consistently with $R$ being in $\Sigma^{-1}_{a}$.

\subsubsection*{The construction}
The construction is in stages: at stage $s$ we define the approximation
$R[s]$ to $R$, and the approximations to the various parameters $(x_{e},
y_{e})$. We will often omit to mention the stage to which a given parameter
is referred, if this is clear from the context. Unless otherwise specified,
at each stage each parameter keeps the same value as at the previous stage.
To \emph{initialize} a $Q$-requirement at stage $s$ means to cancel at that
stage the current value of its witness. A pair $(x_e,y_e)$ is an
\emph{active} witness for $Q_e$ at stage $s$ if the pair has been appointed
as a witness for $Q_e$ at some previous stage and never canceled thereafter.
We say that a requirement $P_e$ is \emph{inactive} at the end of stage $s$ if
there are already distinct numbers $u,v \in W_e$ such that $u \rel{R} v$ at
the end of stage $s$; it is \emph{active} otherwise.

A requirement $T$ \emph{requires attention at stage $s+1$} if either
\begin{enumerate}
 \item $T$ is initialized; or
 \item one of the following holds, for some $e$:
 \begin{enumerate}
   \item $T=P_e$, $P_e$ is active at the end of $s$, and at the current
       stage there is a distinct pair of numbers $u,v \in W_{e}$ both
       even or both odd, and such that $u,v$ are bigger than all numbers
       in the union of all current equivalence classes of numbers $i \leq
       x_{e}+y_{e}$;
   \item $T=Q_e$ and $f(\langle x_e, y_e\rangle,s)=f_e(\langle x_e,
       y_e\rangle,s+1)$, where $(x_e,y_e)$ is the witness of $T$ at the
       end of stage $s$.
 \end{enumerate}

\end{enumerate}

\subsubsection*{Stage $0$}
Initialize all $Q$-requirements; no $P$-requirement is inactive; define
$R[0]:= \emptyset$, and consequently $f(x,0):=0$ and $\gamma(x,0):=a$ for all
$x$. (Notice that we do not define here $R[0]:=\Id$, as would be perhaps more
appropriate since we are defining an equivalence relation which is bound to
be reflexive, because we want to accompany the definition of $R$ with an
accompanying $\Sigma^{-1}_a$-approximating pair $\langle f, \gamma\rangle$ to
$R$ so that $f$ must start with $f(x,0)=0$, for every $x$.)

\subsubsection*{Stage $1$}
Define $R[1]:= \Id$, and $f(\langle x, x \rangle, 1):=1$ and $\gamma(\langle
x, x \rangle, 1):=1$, leaving untouched the other values of both $f$ and
$\gamma$. (Recall that $|1|_{\mathcal{O}}=0$. Notice that for every $s \ge
1$, we will have $x \rel{R[s]} x$, so there will never be need to redefine
$\gamma(\langle x, x \rangle, s)$.) Therefore we can say that the
construction essentially starts (with $R:=\Id$) at stage $1$ instead of $0$.

\subsubsection*{Stage $s+1\ge 2$}
Consider the highest priority requirement $T$ that requires attention.
(Notice that there is always such a requirement since at each stage almost
all $Q$-requirements are initialized). Action:
\begin{enumerate}
  \item If $T$ is initialized, then $T=Q_e$ for some $e$: choose a new
      fresh witness $(x_e, y_e)$ for $T$, i.e.\ $x_e=2i$ and $y_e=2i+1$,
      where $i$ is bigger than all numbers $T$-equivalent to numbers so far
      used in the construction (so we can also assume that $e\leq
      x_{e}+y_{e}$). Notice that because of (3) in the definition of a
      $\Pi^{-1}_a$-approximation, we may as well suppose that still
      $f_e(\langle x_e, y_e \rangle,s+1)=1$, $f(\langle x_e, y_e
      \rangle,s)=0$, and $\gamma_e(\langle x_e, y_e
      \rangle,s+1)=\gamma(\langle x_e, y_e \rangle,s)=a$.
  \item otherwise:
  \begin{enumerate}
    \item if $T=P_e$ then pick the least pair $u,v$ as in the definition
        of requiring attention; define $f(\langle u, v\rangle, s+1):=1$
        and $\gamma(\langle u, v\rangle, s+1):=1$; for any other
        $R$-collapse $x \rel{R} y$ induced at this stage by the
        $R$-collapse of $u,v$ define $f(\langle x, y\rangle, s+1):=1$ and
        $\gamma(\langle x, y\rangle, s+1):=1$; declare $P_e$ inactive
        (and thus it will remain inactive forever unless later
        re-initialized);
    \item if $T=Q_e$ then define $f(\langle x_e, y_e \rangle,s+1)\ne
        f_e(\langle x_e, y_e \rangle,s+1)$ and $\gamma(\langle x_e,
        y_e\rangle, s+1):=\gamma_e(\langle x_e, y_e\rangle, s+1)$, so
        that we diagonalize $R$ against $E_e$.
  \end{enumerate}
\end{enumerate}

This is the end of the stage. Initialize all lower priority $Q$-requirements.
Define $R[s+1]$ to be the equivalence relation having $f(\_,s+1)$ as
characteristic function, i.e. (as is easily seen) the equivalence relation
generated by the set of pairs provided by $R[s]$ plus or minus the pairs
enumerated in or extracted from $R$ at $s+1$.

\subsubsection*{The verification}
An easy inductive argument shows that for every requirement $T$ there is a
least stage $t_{T}$ such that no $T'$ of higher priority than $T$, nor $T$
itself, requires attention or acts at any $s \ge t_{T}$. Indeed suppose that
this is true of every $T'$ of higher priority than $T$. Then there is a least
stage $t$ such that no such $T'$ requires attention after $t$. So either
$T=P_e$ and thus $T$ may act at most once after $t$; or $T=Q_e$ and thus $T$
may require attention a first time to appoint the final value of its witness
$(x_e,y_e)$ (this value will never be canceled, because $T$ will never be
re-initialized again, as no higher priority $T'$ will ever act again), and
subsequently finitely many times in response to the finitely many changes of
$f_e(\langle x_e,y_e\rangle,s)$. In any case this shows that $t_{T}$ exists.

\begin{lemma}\label{lem:P-satisfied}
For every $e,s$, if $(x_e,y_e)$ is still an active witness for $P_e$ at $s$,
then at that stage $[x_e]_R \cup [y_e]_R = \{x_e,y_e\}$, each equivalence
class among $[x_e]_R$ and $[y_e]$ being a singleton if and only if
$f_{e}(\langle x_e,y_e \rangle,s)=1$. Moreover, if distinct $u,v$ have the
same parity then $f(\langle u,v\rangle, \_)$ may change at most once from
value $0$ to $1$ in response to some $P$-requirement which becomes inactive.
\end{lemma}

\begin{proof}
Suppose that $(x_e,y_e)$ is active at $s$, and let $t\le s$ be the stage at
which this witness has been appointed for $Q_e$. Then the restraint imposed
by $Q_e$ (reflected in the fact that lower priority $P$-requirements may not
modify the equivalence classes of numbers $i \leq x_{e}+y_{e}$ and thus may
not modify the equivalence classes of $x_{e}$ or $y_{e}$) prohibits any
modification of the equivalence classes of $x_e, y_e$ done by any lower
priority requirement. Such a modification can only be performed by a higher
priority requirement $P_i$, but if such a requirement has acted after $t$,
then the witness $(x_e,y_e)$ has been re-initialized and thus canceled.

The latter claim about the number of changes of $f(\langle u,v\rangle,\_)$ if
$u,v$ are distinct and have the same parity follows from the fact that by the
first part of this lemma the value  $f(\langle u,v\rangle,\_)$ is not changed
by any $Q$-requirement, so if the value $f(\langle u,v\rangle,1)=0$ is later
changed from $0$ to $1$ then this is due to a $P$-requirement, which becomes
inactive, and by choice of $u,v$ and the first part of this lemma this change
is never revoked by any higher priority $Q$-requirement.
\end{proof}

This enables us to show also that $T$ is in the end satisfied. This is
evident if $T=F_{e}$ as after $t_{Q_{e}}$ no $P$-requirement can add numbers
to $[e]_{R}$ since $e\leq x_{e}+y_{e}$. Notice that this shows that each
$R$-equivalence class is finite.

It is also evident if $T=Q_e$ as $T$ is by Lemma~\ref{lem:P-satisfied} the
only requirement which is entitled to move its final witness $(x_e,y_e)$ in
or out of $R$.

To show that $T$ is satisfied if $T=P_e$, assume that $W_e$ is infinite.
Since the witness $(x_{e}, y_{e})$ reaches a limit, and all equivalence
classes are finite, it follows that $W_{e}$ contains at least two even
numbers or two odd numbers bigger than any element in the equivalence class
of some $i \leq x_{e} + y_{e}$ so that at some point we are able to
$R$-collapse such a pair $u,v$ if $P_{e}$ is still active. So if $W_e$
is infinite then it can not be a transversal of $R$.

By Lemma~\ref{lem:P-satisfied} it is also clear that $R\in \Sigma^{-1}_{a}$,
and the pair $\langle f, \gamma\rangle$ is a $\Sigma^{-1}_{a}$-approximation
to $R$. Indeed on pairs $u,v$ such that $u,v$ have the same parity the
characteristic functions of $R$ may change value on $\langle  u, v \rangle$
only once (from $0$ to $1$). On  pairs of different parity the characteristic
function may change value more times. Indeed, it can do so on $u, v$, with
say $u$ even and $v$ odd for the following reasons. A first possibility is
that for instance we already have $z \rel{R} v$ due to some $Q$-requirement
which has used the pair $(z,v)$ as witness, and some higher priority
$P$-requirement $R$-collapses $u$ and $z$ so that we get $u \rel{R} v$ (or we
already have $z \rel{R} u$  for some odd $z$, and now a higher priority $P$
requirement $R$-collapses $v$ and $z$): but this is compatible with the
relation being in $\Sigma^{-1}_{a}$, as witnessed by $\gamma(\langle u,v
\rangle, \_)$ which decreases from $a$ to $1$. A second possibility is that
the characteristic function of $R$ changes on $\langle u,v \rangle$  if this
value is moved by the diagonalizing strategy of some $Q$-requirement $Q_{e}$,
but then this is done with a number of changes compatible with the relation
being in $\Sigma^{-1}_{a}$, as witnessed by $\gamma(\langle u,v \rangle, \_)$
which mimics $\gamma_{e}(\langle u,v \rangle, \_)$.

\medskip

Finally, we consider the case when $\mathcal{X}=\Pi^{-1}_a$ for some notation
$a$ with $|a|_{\mathcal{O}}>1$. The construction is virtually the same as in
the dual case. We start with $R[0]=\Id_1$ with consequent definitions of $f$
and $\gamma$: $f(x, 0):=1$; $\gamma(x, 0):=a$ for every $x$, and $R[1]$
defined as follows:
\[
x \rel{R[1]} y \Leftrightarrow x=y \lor (\exists i)[\{x,y\}=\{2i, 2i+1\}],
\]
with consequent suitable definitions of $f$ and $\gamma$: in particular
\[
\gamma(\langle x, y\rangle,1)=
\begin{cases}
a, &\textrm{if $x=y \lor (\exists i)[\{x,y\}=\{2i, 2i+1\}]$},\\
2, &\textrm{otherwise,}
\end{cases}
\]
where we have used the assumption that $|a|_{\mathcal{O}}>1$. The
construction now mimics the one for the $\Sigma^{-1}_a$-case, essentially
starting from stage $1$. The idea is to make the first approximation to $R$
to look very much like $\Id$ (as in the proof of
Theorem~\ref{thm:properly-dark}), except for pairs of the type $(2i, 2i+1)$
for which we do not want to add an extra initial change which could spoil our
possibility of playing the extraction/enumerating game exploited by  the
diagonalization strategies. This allows a requirement $P_e$ to be able, if it
requires attention, to $R$-collapse pairs of numbers of the same parity that
have been set as non-$R$-equivalent at stage $1$ (when for this purpose we
have defined $\gamma$ to be $2$ for these pairs).
\end{proof}

\section{The problem of the existence of infima}
Let us fix a notation $a$ for a non-zero computable ordinal. The following
observation is straightforward.

\begin{fact}
The poset of degrees of $\Sigma^{-1}_a$-equivalence relations is not a
lower-semilattice.
\end{fact}

\begin{proof}
This follows immediately from the fact that the degrees of ceers forms an
initial segment of the degrees of $\Sigma^{-1}_a$-equivalence relations, and
on the other hand ceers are known not to form a lower-semilattice, see for
instance \cite{andrewssorbi2}.
\end{proof}

We will show however that we can find properly $\Sigma^{-1}_a$-equivalence
relations with no $\inf$, proving that the $c$-degrees of
$\Sigma^{-1}_a\smallsetminus \Pi^{-1}_a$ equivalence relations do not form a
lower-semilattice.

Andrews and Sorbi~\cite{andrewssorbi2} showed that many questions about the
degree structure of ceers can be fruitfully tackled by inspecting the
interplay between light and dark ceers. By lifting our focus to the class of
$\Delta^0_2$ equivalence relations, we need to introduce the following
relativized version of Definition \ref{defn:light}, which stands as a natural
companion of the analysis of the complexity of transversals of a given ceer
provided in \cite{reducibilityspectra}.

\begin{defn}\label{defn:mutuallydark}
Let $R$ be an equivalence relation, and $A$ any set of numbers. $R$ is
\emph{$A$-dark} if $R$ is infinite and has no infinite
\emph{$A$-transversal}, i.e. there is no infinite $A$-c.e. set $W_e^A$ such
that for all distinct $u,v \in W_e^A$, one has $u \cancel{\rel{R}} v$.

Two equivalence relations $R,S$ are \emph{mutually dark} if $R$ is $S$-dark,
and $S$ is $R$-dark.
\end{defn}

Notice:
\begin{prop}\label{prop:mutual-implies-dark}
If $R,S$ are mutually dark then they are dark.
\end{prop}

\begin{proof}
In fact, if $R$ is such that there is a set $A$ such that $R$ has no infinite
$A$-c.e. set as a transversal then $R$ is dark. This follows from the fact
that a c.e. set is $A$-c.e. relatively to every oracle $A$.
\end{proof}

By Theorem \ref{thm:mutually-dark}, we prove that mutually dark equivalence
relations exist at all levels of the Ershov hierarchy. But before that, let
us offer a nice alternative characterization of mutual darkness. We first
need to relativize computable reducibility in an obvious way (as in
\cite{reducibilityspectra}, where the complexity of $\mathbf{d}$-computable
reductions is considerably explored): Let $\mathbf{d}$ be a Turing degree.
$R$ is  \emph{$\mathbf{d}$-computably reducible} to $S$, denoted
$R\leq_{\mathbf{d}} S$, if there is a total $\mathbf{d}$-computable function
$f$ such that, for all $x,y\in\omega$,
\[
x \rel{R} y \Leftrightarrow f(x)\rel{S}f(y).
\]

\begin{defn}
Define $R|_d S$ if $R,S$ are infinite and $R\nleq_{\deg_T(R)} S$ and
$S\nleq_{\deg_T(S)} R$.
\end{defn}

Hence, $R|_d S$ holds if the information of neither of the two equivalence
relations is enough by itself to compute a reduction into the other. The next
proposition shows that this is the same as asking that $R$ and $S$ are
mutually dark.

\begin{prop}\label{lem:mutually-impies-non-d}
$R |_d  S$ if and only if $R,S$ are mutually dark.
\end{prop}

\begin{proof}
Let $R,S$ be infinite equivalence relations.

Suppose that $R$ is not $S$-dark, and let $g$ be an $S$-computable function
which lists an infinite transversal of $R$. We claim in this case that $S
\leq_{\deg_{T}(S)} R$. Indeed, a suitable $S$-computable function reducing
$S$ to $R$ can be defined by induction as follows: $f(0):=g(0)$; and
\[
f(n+1):=
\begin{cases}
f(i), &\textrm{if $i$ least such that $i\le n$ and $i \rel{S} n+1$},\\
g(n+1), &\textrm{if there is no $i\le n$ such that $i \rel{S} n+1$}.
\end{cases}
\]
In a similar way one can show that if $S$ is not $R$-dark then $R
\leq_{\deg_{T}(R)} S$.

Vice versa suppose that $S \leq_{\deg_{T}(S)} R$, via a function $f \in
\deg_{T}(S)$. As $S$ is infinite, let $g$ be an $S$-computable function
listing an infinite transversal for $S$. Then the function $f \circ g$ is
$S$-computable and lists an infinite transversal of $R$. Thus $R$ is not
$S$-dark. In a similar way, one shows that if $R \leq_{\deg_{T}(R)} S$ then
$S$ is not $R$-dark.
\end{proof}

\begin{thm}\label{thm:mutually-dark}
If $\mathcal{X} \in \{\Sigma^{-1}_a: a \in \mathcal{O}, |a|_\mathcal{O}\ge
1\}\cup \{\Pi^{-1}_a: a \in \mathcal{O}, |a|_\mathcal{O}>1\}$ then there
exist mutually dark equivalence relations having only finite classes and
properly lying in $\mathcal{X}$.
\end{thm}

\begin{proof}
Let us start again with the case $\mathcal{X}=\Sigma^{-1}_a$, with
$|a|_\mathcal{O}\ge 1$. We build equivalence relations $U,V$ satisfying the
following requirements, for every $e \in \omega$:
\begin{align*}
  &F^{U}_{e}: \text{$[e]_{U}$ is finite,}\\
  &F^{V}_{e}: \text{$[e]_{V}$ is finite,}\\
  &P^U_e:  \text{$W_e^U$ is not an infinite transversal for $V$},\\
  &P^V_e:  \text{$W_e^V$ is
       not an infinite transversal for $U$},\\
  &Q^U_e:  U \ne E_e,\\
  &Q^V_e:  V \ne E_e,
\end{align*}
where  $(E_e)_{e \in \omega}$ is a listing of the $\Pi^{-1}_a$-sets, with
$\langle f_{e}, \gamma_{e}\rangle$ a $\Pi^{-1}_a$-approximation to $E_{e}$.
We build $U,V$ via defining $\Sigma^{-1}_a$-approximating pairs $\langle f^U,
\gamma^U \rangle$ and $\langle f^V, \gamma^V \rangle$ to $U, V$ respectively.

The \emph{priority ordering} of the requirements is
\begin{multline*}
F_{0}^{U}< F_{0}^{V}<Q_0^U < Q_0^V <P_0^U< P_0^V < \cdots \\
< F_{e}^{U}< F_{e}^{V} < Q_e^U < Q^V_e < P^{U}_e < P_e^V < \cdots.
\end{multline*}
We say that a requirement $R$ has \emph{priority position $e$}, if $e$ is the
position of $R$ in the above ordering: the priority position of $R$ will be
denoted by $e_{R}$.

\subsubsection*{Strategies for the requirements and their interactions}
The strategies for the requirements are essentially the same as the ones for
the ``corresponding'' requirements in the proof of
Theorem~\ref{thm:properly-dark}. The additional complication is due to the
fact that we have also sometimes to preserve certain Turing-computations. For
this reason we will view the restraint imposed at any stage by a requirement
$R$ as two finite binary strings $r_{R,U}$ and $r_{R,U}$: if $S \in \{U,V\}$
the string $r_{R,S}$ extends $r^-_{R,S}$ (i.e. $r_{R,S} \supseteq r^-_{R,S}$)
which represents the restraint inherited  at that stage by the higher
priority requirements (the string $r^-_{R,S}$ is empty if $R$ is the highest
priority requirement); in turn, $r_{R,S}$ is a string which lower priority
requirements are bound to preserve if $R$ is not re-initialized: this string
automatically becomes $r_{R,S} = r^-_{R',S}$ where $R'$ is the requirement
immediately following $R$ in the priority ordering. The strings $r^-_{R,S}$
and $r_{R,S}$ depend of course on the stage.

The strategy for $R=F^U_e$ (the one for $F^V_e$ is similar) sets up a
restraint requesting that no lower priority requirement change  $[e]_{U}$.

The strategy for $R=Q^U_e$ (the one for $Q^V_e$ is similar) works with a
suitable witness $(x^U_e, y^U_e) \in 2\omega \times (2 \omega+1)$ (consisting
as in the proof of Theorem~\ref{thm:properly-dark} of numbers of different
parity, say of the form $(2i,2i+1)$), and sets up a restraint to preserve the
$U$-equivalence classes of $x^U_e$ and $y^U_e$.

Let us now consider a $P$-requirement $R=P_e^U$ (the case $R=P_{e}^{V}$ is
similar). The strategy for $P^U_e$ consists in seeing if we can define $U,V$
so that there are distinct $u,v\in W_e^{U}$ of the same parity, and $u,v$ can
be $V$-collapsed without injuring higher priority restraints.

$Q$-requirements may be initialized as in the proof of
Theorem~\ref{thm:properly-dark}. A requirement $R$ \emph{requires attention
at stage $s+1$} if either
\begin{enumerate}
 \item $R$ is initialized; or
 \item one of the following holds, for some $e \in \omega$ and $S \in
     \{U,V\}$, where we assume $S=U$, the other case being similar and
     treated by interchanging the roles of $U$ and $V$:
 \begin{enumerate}
   \item $R=P_e^U$, $R$ is \emph{active} at the end of $s$ (i.e. it is
       not already the case that at $s+1$ there are already distinct
       numbers $u,v$ with $u,v \in W_e^{r^-_{R,U}}$ and
       $r^-_{R,V}(\langle u,v \rangle)\downarrow=1$), and there is a pair
       $(\sigma, \tau)$ of strings and a pair $(u,v)$ of distinct numbers
       of the same parity such that:
       \begin{enumerate}
         \item $ r^{-}_{R,U} \subseteq \sigma \subset c_{U[s]}$ (the
             characteristic function of $U[s]$), $u,v\in W_e^{\sigma}$;
         \item $r^{-}_{R,V}\subseteq \tau \subset
             c_{V[s]_{\col{(u,v)}}}$, and $V[s]_{\col{(u,v)}}$ and
             $V[s]$ give the same equivalence classes relatively to all
             $i < e_{R}$ (this is for the sake of not modifying the
             equivalence classes of any such $i$, so that $F^{S}_{i}$
             is not injured), and finally $V[s]_{\col{(u,v)}}$ and
             $V[s]$  give the same equivalence classes relatively to
             all currently defined $x_{i}^V, y_{i}^V$, with $i < e_{R}$
             (this is for the sake of not modifying the equivalence
             classes of any such $x_{i}^V, y_{i}^V$, so that the
             restraint imposed by $Q^{V}_{i}$ is not injured).
       \end{enumerate}

   \item $R=Q_e^U$ and $f^U(\langle x_e^U, y_e^U\rangle,s)=f_e(\langle
       x_e^U, y_e^U\rangle,s+1)$, where $(x_e^U,y_e^U)$ is the witness of
       $R$ at the end of stage $s$.
 \end{enumerate}
\end{enumerate}

\subsubsection*{The construction}
At stage $s$ we define approximations to $f^{S}(\_,s)$ and $\gamma^{S}(\_,s)$
for $S \in \{U,V\}$ ($U[s]$ and $V[s]$ will be the equivalence relations
having $f^{U}(\_,s)$ and $f^{V}(\_,s)$ as characteristic functions,
respectively), and the approximations to the various parameters, including
$r_{R,S}$ and $r^{-}_{R,S}$. We will often omit to mention the stage to which
a given parameter is referred, if this is clear from the context. Unless
otherwise specified at each stage each parameter keeps the same values as at
the previous stage. It will be clear from the construction that at each stage
$s\ge 1$, $U[s]$ and $V[s]$ are equivalence relations with only finite
equivalence classes and such that almost all equivalence classes are
singletons.

\subsubsection*{Stage $0$}
Initialize all $Q$-requirements; no $P$-requirement is inactive; for $S\in
\{U,V\}$ define $f^{S}(x,0):=0$ and $\gamma^{S}(x,0):=a$ for all $x$. Define
$S[0] := \emptyset$.

\subsubsection*{Stage $1$}
For $S\in \{U,V\}$ define $S[1]:= \Id$, with suitable definitions of
$f^{S}(\_,1)$ and $\gamma^{S}(\_,1)$ (similar to the definitions of $f$ and
$\gamma$ at stage $1$ in the proof of Theorem~\ref{thm:properly-dark}).

\subsubsection*{Stage $s+1\ge 2$}
Consider the highest priority requirement $R$ that requires attention.
(Notice that there is always such a requirement since at each stage almost
all $Q$-requirements are initialized). Action:
\begin{enumerate}
  \item If $R$ is initialized, then $R=Q_e^{S}$ for some $e$ and $S \in
      \{U,V\}$. Assume $S=U$: the other case is similar, and is treated by
      interchanging the roles of $U$ and $V$. Choose a new fresh witness
      $(x_e^{U}, y_e^{U})$ for $R$, i.e.\ $(x_e^{U}, y_e^{U})=(2i,2i+1)$
      where $i$ is bigger than all numbers so far used in the construction:
      we may as well suppose that still $f_e(\langle x_e^U, y_e^U
      \rangle,s+1)=1$, but $f^{U}(\langle x_e^U, y_e^U \rangle,s)=0$, and
      $\gamma_e(\langle x_e^U, y_e^U \rangle,s+1)= \gamma^{U}(\langle
      x_e^U, y_e^U \rangle,s)=a$. Let $r_{R,U}$ be the least string
      extending $r^{-}_{R,U}$, such that $r_{R,U}(\langle x_e^U, y_e^U
      \rangle)\downarrow$, and $r_{R,U} \subseteq U[s]$; let
      $r_{R,V}:=r^{-}_{R,V}$. Let $U[s+1]:=U[s]$, and $V[s+1]:=V[s]$.
        \item Otherwise:
  \begin{enumerate}
    \item $R=P_e^{S}$ for some $e$ and $S \in \{U,V\}$. Suppose that
        $S=U$, the other case being similar, and treated by interchanging
        the roles of $U$ and $V$. Then pick the least quadruple $(\sigma,
        \tau, u, v)$ such that the pairs $(\sigma,\tau)$ and $(u,v)$ are
        as in the definition of requiring attention; declare $P_e$
        inactive (and it will remain inactive as long as the requirement
        is not initialized); define $r_{R,U}:=\sigma$ and $r_{R,V}:=\tau$
        (notice that if eventually $\tau$ is an initial segment of $V$
        then the numbers $u,v$ are $V$-collapsed for the sake of $R$). On
        pairs $\langle x,y\rangle$ that are $V$-collapsed following this
        action define $f^{V}(\langle x,y\rangle,s+1):=1$ and
        $\gamma^{V}(\langle x,y\rangle,s+1):=1$. Define $U[s+1]:=U[s]$
        and $V[s+1]:=V[s]_{\col{(u,v)}}$.

    \item $T=Q_e^{S}$: assume that $S=U$, the other case being similar,
        and treated by interchanging the roles of $U$ and $V$. Define
        $f^{U}(\langle x^{U}_e, y^{U}_e \rangle,s+1) \ne f_e(\langle
        x^{U}_e, y^{U}_e \rangle,s+1)$ and $\gamma^{U}(\langle x^{U}_e,
        y^{U}_e\rangle, s+1):=\gamma_e(\langle x^{U}_e, y^{U}_e\rangle,
        s+1)$, so that we diagonalize $U$ against $E_e$ at witness
        $(x_e^U,y_e^U)$.  Define $V[s+1]:=V[s]$ and $U[s+1]$ to be the
        equivalence relation such that $x_e^U \rel{U[s+1]} y_e^U$ if and
        only if $x_e^U \cancel{\rel{U[s]}} y_e^U$ and coinciding with
        $U[s]$ on all other pairs. Set $r_{R,V}:=r_{R,V}^-$, and define
        $r_{R,U}$ to be the least string defined on $\langle x_e^U,
        y_e^U\rangle$ and such that $r_{R,V}^- \subseteq r_{R,U} \subset
        c_{U{s+1}}$.
  \end{enumerate}
\end{enumerate}
At the end of the stage, initialize all lower priority $Q$-requirements.

\subsubsection*{The verification}
A straightforward inductive argument shows again that for every requirement
$R$ there is a least stage $t_{R}$ such that no $R'$ of higher priority than
$R$ nor $R$ itself requires attention or acts at any $s \ge t_{R}$. Hence all
parameters for $R$, including witnesses $(x_{i}^{S}, y_{i}^{S})$ and the
restraints $r_{R,S}, r^-_{R,S}$, for $S \in \{U,V\}$, reach a limit.

An argument similar to Lemma~\ref{lem:P-satisfied} enables us to conclude
that all $U$- and $V$-equivalence classes are finite and each $R$ is
satisfied. Let us check this in the particular case $R=P_e^U$. First of all
notice that whenever we $V$-collapse two numbers $u,v$ for the sake of $R$,
then the $V$-classes of these two numbers will never be separated again:
indeed, no future choice of a string $\sigma=r_{R',V}$, for some
$P$-requirement $R'=P_{i}^V$, can do this when choosing $\sigma$ so that
$u',v' \in W_i^\sigma$, for some pair $u',v'$, because such a $\sigma$ is
chosen so that $T_\sigma$ is the same equivalence relation as at the previous
stage and thus it does not introduce any new change in the corresponding
characteristic function; on the other hand no $Q$-requirement can separate
the $V$-classes of $u,v$ because by the analogue of
Lemma~\ref{lem:P-satisfied}, $Q$-requirements may only move pairs of numbers
of different parity. Now, if $W_e^{U}$ is infinite then by the finiteness of
the $V$-equivalence classes, it enumerates a distinct pair $u,v$ of the same
parity which are not restrained from being $V$-collapsed by higher priority
requirements (which request not to modify the finitely many $V$-equivalence
classes they use). Then there certainly are $\sigma, \tau$ and a stage $t$
such that $R$ is never initialized after $t$, $\sigma$ is an initial segment
of the characteristic function of $U$ and $u,v \in W_{e}^{\sigma}$ at any
stage $s \ge t$, $\tau(\langle u,v \rangle)=1$ and $\tau$ is
$(R,V)$-compatible at any stage $s \ge t$. Then at any such stage $s \ge t$
the pairs $(\sigma, \tau)$ and $(u,v)$ are eligible to make $P_{e}$ require
attention if $P_e$ is still active. In this case the action requested by the
construction at such a stage makes $u,v \in W_{e}^{U}$ and $u \rel{V} v$,
achieving that $W_{e}^{U}$ is not an infinite transversal for $V$.

An argument similar to Lemma~\ref{lem:P-satisfied} enables us also to
conclude that $U,V \in \Sigma^{-1}_{a}$ and that the pairs $\langle f^{U},
\gamma^{U}\rangle$ and $\langle f^{V}, \gamma^{V}\rangle$ are
$\Sigma^{-1}_{a}$-approximations to $U, V$, respectively.

\smallskip

Finally,  the case when $\mathcal{X}=\Pi^{-1}_a$ for some notation $a$ with
$|a|_{\mathcal{O}}>1$ is treated exactly as in
Theorem~\ref{thm:properly-dark}.
\end{proof}

As anticipated by Remark \ref{rem:theorems}, Theorem \ref{thm:properly-dark}
immediately follows from Theorem \ref{thm:mutually-dark} and Proposition
\ref{prop:mutual-implies-dark}.

\smallskip

Our goal now is to prove that no pair of mutually dark equivalence relations
can have infimum. To show this, we make use of the operation $R \mapsto R
\oplus \Id_1$, that has been greatly exploited in \cite{andrewssorbi2}. This
operation can be viewed as an inverse of the operation on equivalence
relations obtained by collapsing two equivalence classes, and leading from an
equivalence relation $R$ to $R_{\col(x,y)}$. In fact, if $R$ is an
equivalence relation and $z$ is a number such that $[z]_R$ is not a
singleton, then define $R_{[z]}$ to be the equivalence relation
\[
x \rel{R_{[z]}} y \Leftrightarrow x =y \lor
[x\rel{R}y \,\&\, z \notin \{x,y\}].
\]
So we see that in getting $R_{[z]}$ instead of collapsing two equivalence
classes we do exactly the opposite, i.e. splitting an equivalence class into
two classes. We have:

\begin{lemma}
For every $z$ such that $[z]_R$ is not a singleton, $R\oplus \Id_1 \equiv_c
R_{[z]}$.
\end{lemma}

\begin{proof}
To show $R_{[z]} \leq_c R\oplus \Id_1$, consider the computable function $f$
where
\[
f(x)=
\begin{cases}
1, &\text{if $x=z$},\\
2x,&\text{if $x \ne z$}.
\end{cases}
\]
To show $R\oplus \Id_1 \leq_c R_{[z]}$, pick $y\ne z$ in the equivalence class
of $z$, and consider the computable function $g$,
\[
g(x)=
\begin{cases}
\frac{x}{2}, &\text{if $x$ even and $\frac{x}{2} \ne z$},\\
y, &\text{if $x$ even and $\frac{x}{2}= z$},\\
z,&\text{if $x$ is odd}.
\end{cases}
\]
\end{proof}

The following lemma has been proved in \cite{andrewssorbi2} for ceers (see
Observation~4.2 and Lemma~4.6 of \cite{andrewssorbi2}), but the same proof
works whatever equivalence relation $R$ one starts with.

\begin{lemma}\label{lem:singleton}
If $R$ is dark then $R< R \oplus \Id_1$.
\end{lemma}

\begin{proof}
We recall how the proof goes. First of all one shows that if $R$ is dark then
$R$ is \emph{self-full} i.e. any reduction $f:R\leq_c R$ must have range that
intersects all equivalence classes: indeed, if $f$ were a reduction missing
say the equivalence class of $a$, then the orbit $\orb_f(a)$ would be easily
seen to provide a transversal for $R$. Next, one easily shows that if $S$ is
a self-full equivalence relation then $S<_c S\oplus \Id_1$ (in fact it is shown
in \cite{andrewssorbi2} that $S$ is self-full if and only if $S<_c S\oplus
\Id_1$).
\end{proof}

\begin{lemma}\label{lem:no-inf}
If $R,S \in \Sigma^{-1}_a$ and $R |_d S$ then $R,S$ do not have $\inf$ in the
$\Sigma^{-1}_a$-equivalence relations.
\end{lemma}

\begin{proof}
Let $R,S \in \Sigma^{-1}_a$ be such that $R |_dS$, and suppose that $T$ is an
infimum of $R,S$ in  $\Sigma^{-1}_a$, i.e. $T \leq_c R,S$ and for every $Z$
such that $Z\leq_c R,S$ we have $Z \leq_c T$. It is clear that $T$ is not
finite. Let $f,g$ be computable functions reducing $T$ to $R$ and $S$. We
claim that $f$ and $g$ do not hit respectively all the $R$-classes and all
the $S$-classes, i.e. there exist numbers $y_{R}, y_{S}$ such that for every
$x$, $f(x) \cancel{\rel{R}} y_{R}$ and $g(x) \cancel{\rel{S}} y_{S}$,
respectively. Suppose for instance that for every $y$ there exists $x$ such
that $f(x) \rel{R} y$. Then it is easy to define an $R$-computable function
$f^{*}$ reducing $R\leq_c T$: just set $f^{*}(y)=x$ where $x$ is the least
number such that $f(x) \rel{R} y$. It follows that $g\circ f^{*}$ is an
$R$-computable function reducing $R\leq_{\deg_{T}(R)}S$ contradicting that
$R|_{d}S$. In a similar way one shows that the range of $g$ avoids some
$S$-classes.

We now derive a contradiction by showing that $T\oplus \Id_1 \leq_c R,S$ and
applying Lemma \ref{lem:singleton}. We use the existence of $y_{R}$ to show
that a suitable reducing computable function $f^{-}$ is given by
\[
f^{-}(x)=
\begin{cases}
f(x), &\textrm{if $x$ even},\\
y_{R}, &\textrm{if $x$ odd}.
\end{cases}
\]
A similar argument shows that $T \oplus \Id_1 \leq_c S$.
\end{proof}

We are now in a position to prove:

\begin{thm}
For $a\in \mathcal{O}$ such that $|a|_{\mathcal{O}}>0$, there are two properly
$\Sigma^{-1}_{a}$ equivalence relations without infimum.
\end{thm}

\begin{proof}
By Theorem~\ref{thm:mutually-dark} and Lemma~\ref{lem:mutually-impies-non-d}
let $R,S$ be two equivalence relations lying properly in $\Sigma^{-1}_{a}$
such that $R|_{d}S$. Then by Lemma~\ref{lem:no-inf} $R$ and $S$ have no
$\inf$.
\end{proof}

Having shown that  no pair of mutually dark $c$-degrees can have infimum, it is
natural to ask whether one can obtain the same with dark equivalence
relations. The next result answers negatively this question: there are
infinitely many levels of the  Ershov hierarchy which properly contain a pair
of dark equivalence relations with infimum.  This contrasts to the case of
ceers (where no pair of dark ceers have infimum, see \cite{andrewssorbi2})
and thus vindicates the idea that mutual darkness is the correct analogous of
darkness for $\Delta^0_2$ equivalence relations.

\begin{thm}\label{thm:dark-with-inf}
\begin{enumerate}
\item There are dark equivalence relations $R,S \in \Pi^{-1}_2$ such that
    $R$ and $S$ have an infimum.
		
\item For every $a\in\mathcal{O}$ such that $|a|_{\mathcal{O}}>0$, there are dark $\Delta^0_2$ equivalence
    relations $E,F\notin \Sigma^{-1}_a$ such that $E$ and $F$ have an
    infimum.
\end{enumerate}
\end{thm}

\begin{proof}
(1)\ Let $(\mathbf{x}_m,\mathbf{y}_m)$ be a minimal pair of c.e.
$m$-deg\-re\-es. We choose c.e. sets $X\in\mathbf{x}_m$ and
$Y\in\mathbf{y}_m$. We also choose a dark ceer $Q$.

We will show that the equivalence relations $R := F_X\oplus Q$ and $S :=
F_Y\oplus Q$ have infimum $T = \Id_2\oplus Q$, where $F_X,F_Y$ are as in
Lemma~\ref{lemma:F_X}. It is clear that $T$ is a lower bound for $R$ and $S$.
Furthermore, by Lemma~\ref{lemma:F_X} both $R$ and $S$ are $\Pi^{-1}_2$
relations. In addition, the darkness of $Q$ ensures that $R$ and $S$ are
dark.

Assume that $E$ is a lower bound of $R$ and $S$. Consider a reduction
$g\colon E \leq_c R$. Then exactly one of the following three cases holds:
\begin{itemize}
\item[(a)] $\range(g)$ does not contain even numbers. Then the
    function $g_1\colon x \mapsto [g(x)/2]$ is a reduction from $E$ to $Q$,
    and we have $E\leq_c Q \leq_c T$.
	
\item[(b)] There is only one class $[w]_E$ such that $g([w]_E)\subseteq
    2\omega$. Then the set $[w]_E$ is computable, and the function
\[		g_2(x) := \left\{
			\begin{array}{ll}
				0, & \text{if~} x\rel{E} w,\\
				g(x), & \text{otherwise};
			\end{array}
		\right.
\]
gives a reduction from $E$ to $\Id_1\oplus Q$. Thus, $E\leq_c \Id_1
\oplus Q \leq_c T$.
	
\item[(c)] There are two different classes $[u]_E$ and $[v]_E$ such that
    $g([u]_E\cup [v]_E)\subseteq 2\omega$. Note that if $x\cancel{E} u$ and
    $x\cancel{E} v$, then $g(x)$ must be an odd number.

    We distinguish two cases. If the class $[u]_E$ is computable, then it
    is easy to show that $E\leq_c \Id_2\oplus Q$.	Assume that $[u]_E$ is
    non-computable. Without loss of generality, suppose that $[u]_E$ is a
    co-c.e. set and $[v]_E$ is c.e. Hence, $[u]_E\leq_m \overline{X}$.
    Recall that $E\leq_c S$, and the relation $S$ contains only one
    non-computable co-c.e. class, namely, the $S$-class $\{ 2z \,\colon
    z\in \overline{Y}\}$. Thus, we deduce that $[u]_E \leq_m \overline{Y}$.
    Hence, the choice of the sets $X$ and $Y$ implies that the set $[u]_E$
    must be computable (as its complement would be $m$-re\-du\-cib\-le to
    both $X$ and $Y$, and thus would be computable), which gives a
    contradiction.
\end{itemize}
In each of the cases above, we showed that $E\leq_c T$, therefore, $T$ is the
greatest lower bound of $R$ and $S$.

(2) The proof of the second part is similar to the first one, modulo the
following key modification: One needs to choose $\Delta^0_2$ sets $X$ and $Y$
such that $X,Y \not \in \Sigma^{-1}_a$ and the $m$-degrees $\deg_m(X)$ and
$\deg_m(Y)$  form a minimal pair. The existence of such sets is guaranteed by
the following more general theorem: we prove that, in any $\Sigma$-level of
the Ershov hierarchy that corresponds to a successor ordinal (i.e., having
notation $2^{a}$ for some $a$), there are $T$-degrees that form a minimal
pair and do not  contain any set of a lower $\Sigma$-level. Note also that
any minimal pair with respect to $\leq_T$ is also a minimal pair with
respect to $\leq_m$.

\begin{thm}\label{minpair}
For every notation $a\in\mathcal{O}$, there are $\Sigma^{-1}_{2^a}$ sets
$X,Y$ such that $\deg_T(X)$ and $\deg_T(Y)$ form a minimal pair and do not
contain $\Sigma^{-1}_a$ sets.
\end{thm}

\begin{proof}

This theorem is a combination of the two following results. The first one is
Selivanov's result \cite{Seliv1985} about properness of every level in the
Ershov hierarchy relative to Turing reducibility. The second one is Yates'
construction of a minimal pair of c.e.\ Turing degrees (see, e.g.,
\cite{Soare:Book}). Below we sketch how these two constructions can be
combined together.

We satisfy the following infinite sequence of requirements (recall that, by
Posner's trick, we can consider the same index in an $N$-requirement, see
\cite{Soare:Book}):

\begin{align*}
  &N_e: \Phi_e^X = \Phi_e^Y = f \text{\ \ is total}
  \Rightarrow f \text{\ \ is computable},\\
  &Q^X_e: X \neq \Psi_e^{E_e} \vee E_e \neq \Theta_e^X,\\
  &Q^Y_e: Y \neq \Psi_e^{E_e} \vee E_e \neq \Theta_e^Y,
\end{align*}
where $\{ \Phi_e \}_{e \in \omega}$ is an effective list of all Turing
functionals, and $\{ \Psi_e, \Theta_e, E_e \}_{e \in \omega}$ is an effective
list of all possible triples consisting of a pair of Turing functionals
and a $\Sigma^{-1}_a$ set. As before, we consider $\langle f_{e},
\gamma_{e}\rangle$ as a $\Sigma^{-1}_a$-approximation to $E_{e}$. We also
build pairs $\langle f^X, \gamma^X \rangle$ and $\langle f^Y, \gamma^Y
\rangle$ in order to get $X$ and $Y$.

\subsubsection*{Strategies for the requirements and their interactions}
In the following we freely adopt language and terminology (length-agreement
functions, tree of strategies, use functions, etc.: in particular
a small Greek letter denotes the use function of a
Turing functional denoted by the corresponding capital Greek letter) which belong to the
jargon of the (infinite) priority method of proof: for details the reader is
referred to \cite{Soare:Book}.

A $Q$-requirement in isolation can be satisfied using Cooper's idea in
\cite{CooperPhD71}, as adapted by Selivanov \cite{Seliv1985} to the
infinite levels of the Ershov hierarchy. Without loss of generality, we
consider the strategy for $Q^X_e$, also we assume that each functional is
nondecreasing by stage and increasing by argument. For the sake of
convenience consider the following length-agreement function:
\begin{multline*}
l(X,e,s) = \mu z ( \forall x \leq z \ \ (X(x)[s] = \\
\Psi_e^{E_e}(x)[s] \wedge \Theta^X_e \upharpoonright
\psi_e(x)[s] = E_e \upharpoonright \psi_e(x)[s])).
\end{multline*}
Thus the strategy for $Q^X_e$ works as follows:

\begin{itemize}

\item[(1)] Choose a ``big'' witness $x_e = x $ at stage $s_0$, thus
    $f^X(x,s_0) =0$ and $\gamma^X(x,s_0) =2^a$.

\item[(2)] Wait for a stage $s_1 > s_0$ such that $x \leq l(X,e,s_1)$.

\item[(3)] ``Put'' $x$ into $X$, namely define $f^X(x,s_1+1) :=1$ and
    $\gamma^X(x,s_1+1) :=a$.

\item[(4)] Wait for a stage $s_2 > s_1$ such that $x \leq l(X,e,s_2)$.

\item[(5)] Thus at stage $s_2$ we have that $f_e(z,s_1) \neq f_e(z,s_2)$
    and $a \geq_\mathcal{O} \gamma_e(z,s_1) >_\mathcal{O} \gamma_e(z,s_2)$
    for some element $z < \psi_e(x)[s_1]$. Moreover, from this stage on we
    have (in case $X$ changes only at $x$) that if $x \leq l(X,e,s)$ then
    $f^X(x,s) =0$ if and only if $f_e(z,s) = f_e(z,s_1)$. This means that
    by putting and extracting $x$ we force $z$ to go in and out from $E_e$.

\item[(6)] Thus, we define $f^X(x,s_2+1) :=0$ and $\gamma^X(x,s_2+1) :=
    \gamma_e(z,s_2)$.

\item[(7)] At later stages $s$, if we see that $x \leq l(X,e,s)$ then
    define $f^X(x,s+1) := 1 - f^X(x,s) $ and $\gamma^X(x,s+1) :=
    \gamma_e(z,s)$. Clearly, if later $f_e(x,t) \neq f_e(x,s)$ for some
    $t>s$ then $\gamma_e(z,t) <_\mathcal{O} \gamma_e(z,s)$. Thus, at stage
    $t$, if $x \leq l(X,e,t)$ then we can act the same as at stage $s$ and
    define (in particular) $\gamma^X(x,t+1) = \gamma_e(z,t) <_\mathcal{O}
    \gamma^X(x,s+1)$.

\end{itemize}

Notice:

\begin{itemize}

\item If $Q^X_e$ can keep $X$ restrained below $\theta_e(\psi_e(x))$ then
    it is enough for the win.

\item The function $\gamma^X(x,s)$ always has the possibility to be defined
    as notation of a smaller ordinal unless $\gamma_e(z,s) = 1$ (note that
    when $\gamma_e(z,s)$ turns into 1 the function $\gamma^X(x,s)$ has the
    possibility to have a last change, thus after this change for any $t >
    s$ we never see $x < l(X,e,t)$ and the strategy becomes satisfied).

\end{itemize}
Therefore, each $Q^X_e$-strategy changes $X(x)$ finitely many times and wins
(if some initial part of $X$ is restrained). The restraint can easily be
achieved by initialization of lower priority strategies. The strategy for
$Q^Y_e$ works in the same way. Thus, each $Q$-strategy is a finitary strategy
and has only one outcome $fin$ on the tree of strategies.

An $N$-requirement in isolation is satisfied by waiting for an expansionary
stages and restraining either $\Phi_e^X$ or $\Phi_e^Y$. The definition of an
$e$-expansionary stage is as usual. Namely, let the length-agreement function
be defined as follows:
\[
l(e,s) = \mu z ( \forall u \leq z\ \
\Phi_e^{X}(u)[s]\downarrow = \Phi_e^{Y}(u)[s]\downarrow),
\]
and define a stage $s$ to be $e$-expansionary if $l(e,s) > l(e,t)$ for all $t
< s$. The goal of a strategy $N_e$ is to build a computable function: thus
for a given $u$ it waits for the first $e$-expansionary stage which covers
$u$. At this stage $s$ the values of $\Phi_e^{X}(u)[s]\downarrow$ and
$\Phi_e^{Y}(u)[s]$ are the same as $\Phi_e^{X}(u)$ and $\Phi_e^{Y}(u)$
(unless $N$ is initialized). Thus, each $N$-strategy has two outcomes $\infty
< fin$ on the tree of strategies, where it is satisfied vacuously below
outcome $fin$ and build a computable function below outcome $\infty$.

The tree of strategy is a subtree of $\{ \infty < fin \}^{< \omega}$ with the
usual ordering of nodes. At level $k=2e$ we put copies of the strategy $N_e$,
at level $k=4e+1$ we put copies of $Q^X_e$, and at level $k=4e+3$ we put
copies of $Q^Y_e$.

Similar to the construction of a minimal pair we analyze the most problematic
case of interaction between strategies, namely the work of an $N$-strategy
with several $Q$-strategies below its infinite outcome. So, let
$\eta^{\smallfrown}\infty \subset \alpha_1 \subset \alpha_2 \subset \cdots
\subset \alpha_k$. Assume that each $Q$-strategy $\alpha_i$ has current
witness $x_i$. Then it holds that $x_i < \theta_i(\psi_i(x_i)) < x_{i+1} <
\theta_{i+1}(\psi_{i+1}(x_{i+1}))$ for $1 \leq i < k$. Now, if $\alpha_i$
acts at stage $s$ (in particular, either $f^X(x_i,s)$ or $f^Y(x_i,s)$ is
changed) then $x_{i+1}$ and all greater witnesses are canceled. Also we can
visit any of these nodes $\alpha_j$, where $j<i$, only when we get the next
$\eta$-expansionary stage $t>s$, which means that $\Phi^X_e(u)[s] \downarrow
= \Phi^Y_e(u)[s] \downarrow = \Phi^Y_e(u)[t] \downarrow = \Phi^X_e(u)[t]
\downarrow$ for any $u < l(e,s)$. Therefore, if some $\alpha_j$, where $j<i$,
acts at stage $s$ then $\eta$ continues to win, moreover $\alpha_i$ is
initialized and $\alpha_j$ continues to win too since $\theta_j(\psi_j(x_j))$
was smaller than $x_i$ (note also that it does not matter whether $\alpha_j$
works with $X$- or $Y$-side).

Below we sketch construction and verification,  which can easily be expanded
to more formal versions.

\emph{The construction.}  At stage $s+1$ we construct a computable
approximation (with the help of substages) of the true path on the tree of
strategies. Starting from the root node we perform actions relative to the
visited node and decide its outcome, then we visit the node below the outcome
and continue until we reach the node of length $s$. When proceeding to the
next stage, we initialize all nodes to the right to, or below, the visited
one.

If we work with an $N$-strategy $\eta$ at substage $t+1$ then we check
whether the stage $s+1$ is $\eta$-expansionary. If it is so then $\eta$ has
outcome $\infty$, otherwise $\eta$ has outcome $fin$.

If we work with a $Q^X$-strategy (the case of a $Q^Y$-strategy is totally
similar) $\alpha$ at substage $t+1$ then we assign a big witness $x_{\alpha}
= x$ (bigger then any number mentioned so far) and initialize all nodes below
$\alpha$. If witness $x$ was already assigned then check whether $x \leq
l(X,\alpha,s)$. If the answer is ``no'' then just take outcome $fin$; if the
answer is ``yes'' then initialize all nodes below $\alpha$ and also define
$f^X(x,s+1) :=1$ and $\gamma^X(x,s+1) :=a$; moreover if the answer ``yes''
already happened at least one time (after assigning $x$) then there is the
least $z$ such that $f_{\alpha}(z,s_1) \neq f_{\alpha}(z,s_0)$ and
$\gamma_{\alpha}(z,s_1) <_\mathcal{O} \gamma_{\alpha}(z,s_0)
\leq_{\mathcal{O}} a$, where $s_0+1$ is a stage such that $\gamma^X(x,s_0+1)
=a$ and $\gamma^X(x,s_0) =2^a$ (namely, it is the stage when we initiated the
attack using $x$) and $s_1+1$ is the next stage with answer ``yes'', then
define $f^X(x,s+1) :=1 - f^X(x,s)$ and $\gamma^X(x,s+1) :=
\gamma_{\alpha}(z,s)$, and initialize all strategies below $\alpha$. In all
of the cases the outcome is $fin$.

\emph{The verification.} The true path $TP$, defined as the leftmost path on
the tree of strategies visited infinitely often, clearly exists. It remains
to show that each requirement is satisfied by a strategy on the true path.
However, most of the arguments have been already presented in the previous
discussions. Also, by induction it is easy to see that each strategy is
initialized and initializes other strategies only finitely many times.

If $\eta \in TP$ is an $N$-strategy then it clearly satisfies the
corresponding requirement. Namely, waiting for $\eta$-expansionary stages
allows to correctly see the value of the computation (also we assume that it
happens after stage $s_0$ after which $\eta$ is not initialized). If $\alpha
\in TP$ is a $Q^X$-strategy, then let $s_0$ be a stage after which $\alpha$
is not initialized. Henceforth, after assigning the witness $x$ the strategy
$\alpha$ either sees $x > l(X,\alpha,s)$, or several times it sees $x \leq
l(X,\alpha,s)$ (which immediately forces $f^X(x,s)$ to be changed by the
construction). As previously argued, the case $x \leq l(X,\alpha,s)$ can
happen only finitely many times (we always have $\gamma_{\alpha}(z,s) <
\gamma^X(x,s)$ which allows at least one more change for $\gamma^X(x,s)$).
\end{proof}

Going back to the proof of item (2) of Theorem~\ref{thm:dark-with-inf}, we
can now reason as in $(1)$ by using the sets $X,Y$ constructed in the
previous theorem.
\end{proof}

\section{The problem of the existence of suprema}
In this section we consider the problem of when  a given pair of $\Delta^0_2$
equivalence relations has a supremum. The problem  is somehow more delicate
than that the one concerning infima, discussed in the previous section. This
is because a priori the existence of a supremum depends on the level of the
Ershov hierarchy that we choose to consider: E.g., consider $R,S\in
\Sigma^{-1}_a$ and $a<_{\mathcal{O}} b$. Since $R\oplus S \in \Sigma^{-1}_a$
is an upper bound of $\{ R,S\}$, if $T$ is a supremum of $R$ and $S$ at the
$\Sigma^{-1}_b$ level, then $T\in \Sigma^{-1}_a$ as well and $T$ is the
supremum at the $\Sigma^{-1}_a$ level. On the other hand, $R$ and $S$ can
have $\sup$ inside $\Sigma^{-1}_a$, but not have $\sup$ inside
$\Sigma^{-1}_b$. Therefore, we split the problem in two parts: we first
provide a necessary condition for the nonexistence of suprema in
$\Delta^0_2$, and then we restrict the focus to the $\Sigma^{-1}_a$
equivalence relations.

\subsection{Nonexistence of suprema in $\Delta^0_2$}
The next theorem shows that mutual darkness forbids the existence of suprema
in $\Delta^0_2$.

\begin{thm}\label{theo:mut-dark-no-sup}
If  $R,S\in \Delta^0_2$ are mutually dark, then they have no $\sup$ in
$\Delta^0_2$.
\end{thm}

\begin{proof}
Towards a contradiction, assume that $T\in \Delta^0_2$ is the supremum
of $R,S$. We build $U \in \Delta^0_2$ such that $U$ is an upper bound of
$R,S$ and avoids the upper cone of $T$. That is to say, we build $U$
satisfying the following requirements

\begin{align*}
  &P_e:  \text{$\phi_e$ does not reduce $T$ to  $U$},\\
  &Q^R:  R \leq_c U,\\
  &Q^S:  S \leq_c U.
\end{align*}

\subsubsection*{Strategy for the requirements}

Call the \emph{$R$-part} (resp.\ \emph{$S$-part}) of $R\oplus S$ the one
consisting of all even (odd) numbers. Our idea is to construct $U\supseteq R
\oplus S$ by injectively merging equivalence classes of the $S$-part of $R
\oplus S$ with classes of $R$-part of $R \oplus S$ (and by these actions
meeting the above requirements). Eventually, each equivalence class of $U$
will consist of the union of exactly two equivalence classes of $R \oplus S$,
one belonging to its $R$-part and the other to its $S$-part.

The strategy for meeting a single $P_e$-requirement is straightforward. We
look for a pair of distinct numbers $u,v$ such that
$\phi_e(u)\downarrow=x_e,\phi_e(v)\downarrow=y_e $ and $x_e, y_e$ have
different parity. When found, we distinguish two cases:

\begin{enumerate}
\item  If $u\rel{T}v$, we do nothing and keep $[x_e]_U$ and $[y_e]_U$
    separate;
\item If $u \cancel{\rel{T}} v$, we merge $[x_e]_U$ and $[y_e]_U$ and keep
    them together.
\end{enumerate}

The strategy prevents $\phi_e$ from being a reduction of $T$ to $U$. Indeed,
if $\phi_e$ is total and fails to converge on elements with different parity,
then $\range(\phi_e)$ must be all contained in either the $R$-part  or the
$S$-part of $R\oplus S$. Without loss of generality assume $\range(\phi_e)$
is a subset of the $S$-part of $R\oplus S$. If so, one can easily construct
from $\phi_e$ a computable $f$ reducing $T$ to $S$. But since $R\leq_c T$, we
would have $R\leq_c S$, contradicting the fact that $R$ and $S$ are
incomparable.

To deal with $Q$-requirements, we let $U$ be initially equal to $R \oplus S$.
So at first stage of the construction we have that $R\leq_c U$ via $\lambda
x.2x$ and $S \leq_c U$ via $\lambda x.2x+1$. We claim that this is not injured
by the subsequent action of any $P$-requirement. To see why, consider for
instance $R$. $R$ is obviously reducible to $R \oplus S$ via $\lambda x.2x$.
Now let $U$ be obtained from $R \oplus S$ by merging a class of its $R$-part
(say, $[2k]_{R \oplus S}$) with a class of its $S$-part (say, $[2j+1]_{R
\oplus S}$). Then  $R\leq_c U$, since no equivalence classes in the range of
$\lambda x.2x$ have been collapsed. A similar reasoning can be made for $S$.
Hence if we allow $P_e$ to merge  only equivalence classes with different
parity, we obtain that the $Q$-requirements are automatically satisfied.

\subsubsection*{Interaction between strategies}
To combine all strategies, we have to address the following difficulty. When
$P_e$ wants to collapse a given pair of equivalence classes we need to be
careful that this does not imply collapsing also, by transitivity,
equivalence classes of the same parity, because if this happens we might
injure a $Q$-requirement. Suppose for instance that $P_i$ wants to collapse
$[2k]_U$ and $[2j+1]_U$ but $P_e$ already collapsed $[2k]_U$ and $[2i+1]_U$.
If we let $P_i$ be free to act, then by transitivity it would collapse two
equivalence classes of the same parity, i.e.,  $[2j+1]_U$ and $[2i+1]_U$. To
avoid this, we define the following priority ordering of the requirements
\[
P_0 < P_1 < \ldots P_e < \ldots
\]
Next, $P_e$  looks for numbers $u,v$ such that $\phi_e(u)$ and $\phi_e(v)$
have different parity and come from \emph{fresh} equivalence classes of
$R\oplus S$, i.e., neither $u$ nor $v$ belongs to equivalence classes already
collapsed in the construction. A similar condition seems $\Sigma^0_2$: $P_e$
asks whether there exist $u,v$ not being in the union of finitely many
$\Delta^0_2$ sets (i.e., the equivalence classes already collapsed).
Nonetheless, by making use of the fact that $R$ and $S$ are mutually dark we
will prove that $\mathbf{0}'$ can decide such condition (see Lemma
\ref{lem:terminate}).

\subsubsection*{The construction}

We build $U$ in stages, i.e., $U=\bigcup_{k\in \omega} U[k]$. During the
construction we keep track of the equivalence classes that we collapse by
putting a witness them in a set $Z$.

\subsubsection*{Stage $0$}
$U[0]=R\oplus S$ and $Z=\emptyset$.

\subsubsection*{Stage $s+1=2e$}
We deal with $P_e$. In doing so, we execute the following algorithm (which is
computable in $\mathbf{0}'$): List all pairs of distinct numbers $(u,v)$
until one of the following cases happens
\begin{enumerate}
\item $\phi_e(u)\downarrow=x_e,$ $\phi_e(v)\downarrow=y_e$, and
\begin{enumerate}
\item either $u \rel{T} v \not\Leftrightarrow x_e \rel{U} y_e$,
\item or $x_e$ and $y_e$ have different parity and $\set{x_e,y_e}\cap
    Z=\emptyset$;
\end{enumerate}
\item there is $\phi_e(x)\uparrow$.
\end{enumerate}

Lemma \ref{lem:terminate} proves that the algorithm always terminate. If it
terminates with outcome $(1.b)$, let $U[s+1]=U[s]_{\col(x_e,y_e)}$ and put
$[x_e]_{R \oplus S}$ and $[y_e]_{R\oplus S}$ in $Z$; if it terminates with
outcome $(1.a)$ or $(2)$, do nothing.

\subsubsection*{Stage $s+1=2e+1$}
We ensure that eventually any $U$-class will be the merging of an $R$-class
and a $S$-class. To do so, let $u_{e}$ (resp.\ $v_{e}$) be the least even (odd)
number not in $Z$. Let $U[s+1]=U[s]_{\col(u_e,v_e)}$ and put $[u_e]_{R \oplus
S}$ and $[v_e]_{R\oplus S}$ in $Z$.

\subsubsection*{The verification}

The verification is based on the following lemma.

\begin{lemma}\label{lem:terminate}
For all $P_e$, the algorithm defined at stage $2e$ terminates.
\end{lemma}

\begin{proof}
Assume that there is $P_e$ for which the algorithm does not terminate. This
means that $\phi_e$ is total (otherwise, the algorithm at some point would
outcome $(2)$) and  $\phi_e$ reduces $T$ to $U$ (otherwise, at some point the
algorithm would outcome $(1.a)$). Moreover, $\phi_e$ can not hit infinitely
many equivalence classes of both the $R$-part and the $S$-part of $R \oplus
S$.  Otherwise, the algorithm would eventually find a pair of fresh
equivalence classes with different parity. Without loss of generality, assume
that $\phi_e$ hits only finitely many classes in the $R$-part of $R \oplus S$
and let $A=\set{a_0,\ldots,a_n}$ be a transversal of such classes and let
$B=\set{b_0,\ldots,b_n}$ be a set of odd numbers such that, for all $0\leq
i\leq n$, $a_i \rel{U} b_i$. The existence of such $B$ is guaranteed by the
fact that each equivalence class of the $R$-part of $R\oplus S$ is merged in
$U$ with an equivalence class of the $S$-part of $R\oplus S$. But then one can
define the following $\deg(R)$-computable reduction from $T$ to $U$ that hits
only the $S$-part of $U$,
\[
f(x)= \begin{cases}
\phi_e(x) &\text{$\phi_e(x)$ is odd,}\\
b_z, \mbox{where $z$ is such that $xRa_z$}  &\text{otherwise.}
\end{cases}
\]

It follows that $R \leq_{\deg(R)}S$, contradicting the fact that $R|_d S$.
\end{proof}

We are in the position now to show that all $P$-requirements are satisfied.
The last lemma guarantees that, given $P_e$, the corresponding strategy
terminates with either disproving that $\phi_e$ is a reduction from $T$ to
$U$ or by providing two equivalence classes that can be collapsed in $U$ to
diagonalize against $	\phi_e$. The $Q$-requirements are also satisfied
because we carefully avoid, within the construction, to collapse classes of
the same parity.
\end{proof}

By modifying the last proof, we can obtain something stronger.

\begin{thm}
If $R,S \in \Delta^{0}_2$ and $R$ is $S$-dark or $S$ is $R$-dark, then $R,S$
have no $\sup$ in $\Delta^0_2$.
\end{thm}
\begin{proof}[Proof Sketch]
Suppose that $S$ is $R$-dark. Then the proof of
Proposition~\ref{lem:mutually-impies-non-d} shows that $R\nleq_{\deg_T(R)}
S$. 	

Assume that $T$ is a $\Delta^0_2$ equivalence relation, and $T = \sup
\{R,S\}$. We construct a $\Delta^0_2$ equivalence relation $U$ by employing
precisely the same construction as in Theorem~\ref{theo:mut-dark-no-sup}. In
order to verify the construction, it is sufficient to re-prove
Lemma~\ref{lem:terminate} as follows. 	

Suppose that for some $e\in\omega$, the algorithm for $P_e$ does not
terminate. Then, arguing as above, we may assume that:
\begin{enumerate}

\item the function $\varphi_e$ is total,
		
\item $\varphi_e\colon T\leq_cU$, and
		
\item $\varphi_e$ does not hit infinitely many classes of either the $R$-part
    or the $S$-part of the relation $R\oplus S$.
\end{enumerate}
Thus, one of the following two cases holds.
	
\emph{Case 1.}
The function $\varphi_e$ hits only finitely many classes in the $R$-part and
infinitely many classes in the $S$-part. Then the same argument as in the
proof of Lemma~\ref{lem:terminate} shows that $R\leq_{\deg_T(R)} S$, which
contradicts the $R$-darkness of $S$.
	
\emph{Case 2.}
Assume that $\varphi_e$ hits infinitely many classes in the $R$-part and only
finitely many classes in the $S$-part. Then choose a computable function
$h\colon S\leq_cT$, and consider a partial computable function
\[
f(x) :=  \left\{
			\begin{array}{ll}
				\varphi_e(h(x))/2, & \text{if~} \varphi_e(h(x)) \text{~is even},\\
				\uparrow, & \text{otherwise}.
			\end{array}
		\right.
\]
Since the function $\varphi_e\circ h$ reduces $S$ to $U$, the c.e.\ set
$\range(f)$ intersects infinitely many $R$-classes. Therefore, one can choose
an infinite $R$-c.e. set $A$ with the following properties: $A\subseteq
\dom(f)$ and $f(x)\cancel{\rel{R}} f(y)$ for distinct $x,y\in A$. Note that
the condition $f(x)\cancel{\rel{R}} f(y)$ implies that
$\varphi_e(h(x))\cancel{U}\varphi_e(h(y))$, and this, in turn, implies
$x\cancel{S}y$. Hence, $A$ is an $R$-c.e. transversal of $S$, which
contradicts the $R$-darkness of $S$. 	

Therefore, one can re-prove Lemma~\ref{lem:terminate} and verify the
construction.
\end{proof}

The above result contrasts with the fact that $\Id$ and a dark ceer have
always $\sup$ in the ceers: see \cite[Observation~5.1]{andrewssorbi2}.

\subsection{Nonexistence of suprema at the same level of Ershov hierarchy}
We turn now to the problem of whether there are equivalence relations $R,S
\in \Sigma^{-1}_{a}$ with no supremum in $\Sigma^{-1}_{a}$. We know from
\cite{andrewssorbi2} that this is the case for ceers: in particular, there
are light ceers with no sup (in the class of ceers). The next proposition
extends this fact to all levels of the Ershov hierarchy, with the exception
of co-ceers.

\begin{prop}\label{lemma:nosup}
Suppose that $\mathcal{X}\in \{ \Sigma^{-1}_a, \Pi^{-1}_a \,\colon
|a|_{\mathcal{O}} \geq 2\}$. There are light equivalence relations $R$ and
$S$ such that both $R$ and $S$  properly belong to $\mathcal{X}$, and $R,S$
have no $\sup$ in the $\mathcal{X}$-equivalence relations.
\end{prop}
\begin{proof}
Fix a set $A$ properly belonging to the class $\mathcal{X}$. Consider two
c.e. sets $U$ and $V$ such that $U$ and $V$ are $\leq_m$-incomparable. We
define the relations $Q$, $R$, and $S$ as follows:
\begin{gather*}
	Q := \Id \cup \{ (2y,2y+1), (2y+1,2y) \,\colon y\in A\},\\
	R := R_U \oplus Q, \quad S = R_V \oplus Q.
\end{gather*}
It is easy to show that each of the relations $Q,R,S$ is light and properly
belongs to $\mathcal{X}$.

Assume that $T$ is the supremum of $\{ R, S\}$. Without loss of generality,
suppose that $0\in U\cap V$. Since the relation
\[
E :=  (R_U\oplus R_V)_{\col(0,1)} \oplus Q
\]
is an upper bound for $R$ and $S$, we have $T\leq_c E$, and $T$ must be
essentially 1-dimensional. Let $[a]_T$ be the unique non-computable
$T$-class. The conditions $R\leq_c T$ and $S\leq_c T$ imply that
\begin{equation}\label{equ:lem-5.3}
	U\leq_m [a]_T \text{~and~} V \leq_m [a]_T.
\end{equation}

Consider the essentially 2-dimensional relation $F := (R_U \oplus R_V)\oplus
Q$. Since $T$ should be reducible to $F$, we have either $[a]_T\leq_m U$ or
$[a]_T\leq_m V$. This fact and (\ref{equ:lem-5.3}) together contradict the
choice of $U$ and $V$. Thus, $R$ and $S$ have no supremum.
\end{proof}

In the previous proof making $R$ and $S$  properly lying in a given class
$\mathcal{X}$  has as a consequence that $R$ and $S$ are light. Nonetheless,
we can build dark (and even mutually dark) equivalence relations with no
supremum at the same level of the Ershov hierarchy. The next theorem proves
more: we combine what we know of the nonexistence of infima and suprema to
build a pair of equivalence relations that fails to have either.

\begin{thm}\label{thm:omegainf}
There are mutually dark $\omega$-c.e.\ equivalence relations with neither
$\inf$ nor $\sup$ in the degrees of $\omega$-c.e. equivalence relations.
\end{thm}

\begin{proof}
If $U, V$ are $\omega$-c.e. equivalence relations which are mutually dark
then a simplified version  of the argument used in the proof of
Lemma~\ref{lem:no-inf} shows that they have no $\inf$ in the degrees of
$\omega$-c.e. equivalence relations.

If in addition $U,V$ are essentially $1$-dimensional, with one non-computable
class $X_{U}$ for $U$ and $X_{V}$ for $V$, such that $X_{U}|_{m} X_{V}$ then
$U,V$ have no $\sup$ either in the $\omega$-c.e.\ equivalence relations. The
argument for showing this is essentially as in the proof of
Proposition~\ref{lemma:nosup}. Indeed, assume $0 \in X_U \cap X_V$ and
suppose by contradiction that $T$ is such a $\sup$; now, consider the
$\omega$-c.e.\ essentially $1$-dimensional equivalence relation $(U\oplus
V)_{\col(0,1)}$, which is an upper bound of both $U,V$, thus $T
\leq_c(U\oplus V)_{\col(0,1)}$. But then $T$ is essentially $1$-dimensional
too, and thus has exactly one noncomputable equivalence class, say $X_{T}$.
Since $U,V\leq_cT$ it must be $X_{U} \leq_{m} X_{T}$ and $X_{V} \leq_{m}
X_{T}$. On the other hand, as $U\oplus V$ is an upper bound of both $U,V$, we
have that $T \leq_cU\oplus V$, and thus $X_{T} \leq_{m} X_{U}$ or $X_{T}
\leq_{m} X_{V}$, giving $X_{V} \leq_{m} X_{U}$ or $X_{U} \leq_{m} X_{V}$,
contradiction.
\end{proof}

It remains to show that $\omega$-c.e. equivalence relations $U,V$ as the ones
used in the proof of Theorem~\ref{thm:omegainf} exist: this is the goal of the
following lemma.

\begin{lemma}
There exist essentially $1$-dimensional $\omega$-c.e. equivalence relations
$U,V$ which are mutually dark and such that all equivalence classes of $U$
and $V$ are finite with the exception of exactly one $U$-equivalence class
$X_U$ and one $V$-equivalence class $X_V$ for which we have that $X_U |_m
X_V$.
\end{lemma}

\begin{proof}
We build equivalence relations $U, V$ so that the equivalence classes
$X_U:=[0]_U$ and $X_V:=[0]_V$ are such that $X_U |_m X_V$, all other classes
are finite, and $U,V$ are mutually dark.

\subsubsection*{Requirements and strategies}

To achieve our goals, we build $U$ and $V$ to satisfy the following
requirements, for every $e \in \omega$:

\begin{align*}
F^V_e:      &\; \text{$e \cancel{\rel{V}} 0 \Rightarrow [e]_V$
                           finite,}\\
F^U_e:      &\; \text{$e \cancel{\rel{U}} 0 \Rightarrow [e]_U$
                           finite,}\\
P_e^V:      &\; \text{$W_e^V$ is not an infinite transversal
                           for $U$,}\\
P_e^U:      &\; \text{$W_e^U$ is not an infinite transversal
                           for $V$,}\\
I^{U,V}_e:  &\; \text{$\phi_e$ does not $m$-reduce $[0]_U$ to $[0]_V$,}\\
I^{V,U}_e:  &\; \text{$\phi_e$ does not $m$-reduce $[0]_V$ to $[0]_U$}.
\end{align*}
The priority ordering $<$ of the requirements is given by
\[
\cdots < F^V_e< F^U_e < P_e^V < P_e^U
< I^{U,V}_e < I^{V,U}_e < \cdots .
\]
As in the proof of Theorem~\ref{thm:mutually-dark} we say that a requirement
$R$ has priority position $e$, if $e$ is the position of $R$ in the above
ordering: the priority position of $R$ will be denoted by $e_{R}$.

The strategies for these requirements will turn out to be finitary: each
strategy will act only finitely often, and when acting it will modify only a
finite amount of $U$ and $V$. The strategies will be similar to the ones
in the proof of Theorem~\ref{thm:mutually-dark} with the simplification that
we do not have to worry about limiting too much the number of changes in
the characteristic functions of $U$ and $V$.

More precisely, as in the proof of Theorem~\ref{thm:mutually-dark} at any
stage of the construction a requirement $R$ will inherit by the higher
priority requirements two finite binary strings $r^-_{R,U}$ and $r^-_{R,V}$
(these strings are empty if $R$ is the highest priority requirement) which
depend of course on the stage, and are initial segments of the current
approximations of the characteristic functions of $U$ and $V$, respectively:
these strings are the \emph{restraints} which the strategy for $R$ is bound
to preserve. In turn, $R$ will provide its own restraints, i.e. strings
$r_{R,U} \supseteq r^-_{R,U}$ and $r_{R,V} \supseteq r^-_{R,V}$ which lower
priority requirements are bound to preserve.

If $\sigma$ is a finite binary string then we say that $\sigma$ is an
\emph{equivalence string} if the set of pairs $\{(x,y): \sigma(\langle x,y
\rangle) \downarrow) =1\} \cup \Id$ is an equivalence relation: call
$T_{\sigma}$ this equivalence relation. Clearly, $T_{\sigma}$ is an
equivalence relation in which all equivalence classes are finite, and almost
all of them are singletons.

\smallskip

\emph{Strategy for $F^V_e$ and $F^U_e$}. If $T \in \{U,V\}$ then the strategy
for $R=F^T_e$ consists in not allowing lower priority requirements to
$T$-collapse any number to $e$ if at the given stage $e \cancel{\rel{T}} 0$.
\emph{Outcome}: If eventually $e \cancel{\rel{T}} 0$ then $[e]_T$ is finite
as only higher priority requirements can add numbers to this equivalence
class, but each such requirement acts only finitely often. The restraint
strings imposed by the requirement are $r_{R,U}$ and $r_{R,V}$ which equal
the least equivalence strings of length $\ge e_{R}$ which extend
$r^{-}_{R,U}$ and $r^{-}_{R,V}$, respectively.

\smallskip
\emph{Strategy for $P^V_e$ and $P^U_e$}. Consider the case $R=P^V_e$, the
other one being similar. The strategy for $P^V_e$ works within the restraints
imposed by $r^-_{R,U}$ and $r^-_{R,V}$. If $W_e^V$ hits infinitely many
$U$-classes (which is the case if $W_e^V$ is an infinite transversal of $U$)
then by finiteness of the restraints we can pick $a,b$, with $a \ne b$, such
that $a,b$ are not so restrained from $U$-collapsing, and such that we can
redefine $V$ respecting higher priority restraints so that $a, b\in W_e^V$.
More precisely: $R$ waits for numbers $a,b$ and an equivalence string $\sigma
\supseteq r^-_{R,V}$ such that $a \ne b$, $a,b \in W_e^\sigma$ and $a, b
\cancel{\rel{U}} i$ for every $i\leq e_{R}$. \emph{Action}: If and when the
wait is over, requirement $R$ $U$-collapses $a,b$, sets $r_{R,V}$ to be the
least such equivalence string $\sigma$, and $r_{R,U}$ to be the least
equivalence string extending $r^{-}_{R,U}$ so that $r_{R,U}(\langle a, b
\rangle\downarrow) = 1$. \emph{Outcomes}: If the strategy awaits forever,
then $W_e^V$ hits only finitely many $U$-classes, and thus is not a
transversal for $U$; otherwise the action guarantees that $W_e^V$ is not a
transversal for $U$.

\smallskip

\emph{Strategy for $I^{U,V}_e$ and $I^{V,U}_e$}. Let us consider only the
case $R=I_{e}^{U,V}$, the other one being similar. The strategy \emph{acts a
first time} by appointing a new witness $x_{R}$ (being new, it is not  in any
of the finitely many $U$-classes that are restrained by higher priority
requirements, nor is in the current $[0]_U$); waits for $\phi_e(x_{R})$ to
converge, meanwhile restraining $x_{R}$ out of $[0]_U$ (setting $r_{R,U}$ to
be the least equivalence string extending $r^-_{R,U}$ so that
$r_{R,U}(\langle 0, x_{R} \rangle)\downarrow =0$). \emph{Outcome of the first
action}: If the strategy waits forever then $\phi_e$ is not total and thus
the requirement is satisfied. If $\phi_e(x_{R})$ converges then $R$
\emph{acts a second time}: if already $r^-_{R,V}(\langle 0,
\phi_e(x_{R})\rangle) \downarrow =1$ then it keeps the same $r_{R,U}$, and
sets $r_{R,V}=r^-_{R,V}$; otherwise restrains $\phi_e(x_{R})\notin [0]_V$,
and $U$-collapses $0$ and $x_{R}$ (setting $r_{R,U}$ to be the least
equivalence string extending $r^-_{R,U}$ so that $r_{R,U}(\langle 0,
x_{R}\rangle)\downarrow =1$ and $r_{R,V}$ to be the least equivalence string
extending $r^-_{R,V}$ so that $r_{R,V}(\langle 0, \phi_e(x_{R})
\rangle)\downarrow =0$). \emph{Outcome of the second action}: The outcome is
a straightforward diagonalization showing that $\phi_e$ fails to be a
reduction on input $x_{R}$.

\smallskip

As we see, at each stage each requirement $R$ contributes, through $r_{R, U}$
and $r_{R, V}$, finite initial segments to the characteristic functions of
the current approximations to $U$ and $V$, respectively. We will see that
eventually these initial segments stabilize in the limit, and if $S$ has
lower priority than $R$ then the limit strings proposed by $S$ extend those
proposed by $R$: the characteristic functions of the final $U$ and $V$ will
be the unions of these limit finite initial segments proposed by the
requirements.

\smallskip
When we \emph{initialize} a requirement $R$ at a stage $s$ we set as
undefined its parameters $x_{R}$ and its restraints $r_{R,U}$ and
$r_{R,V}$. An $F$-requirement \emph{requires attention} at $s+1$ if it is
initialized; a $P$-requirement \emph{requires attention} at $s+1$ if it is
initialized, or has not as yet acted after its last initialization but is now
ready to act, i.e. we have found suitable $a,b, \sigma$. An $I$-requirement
\emph{requires attention} at $s+1$ if it is initialized, or has not as yet
acted after its last initialization but it is now ready to act either by acting a
first time, or by acting a second time.

\subsubsection*{The construction}
The construction is in stages. Unless otherwise specified at each stage each
parameter keeps the same value as at the previous stage. We will often omit
to mention the stage to which a given parameter is referred, if this is clear
from the context.

\smallskip

\emph{Stage $0$}. Initialize all requirements, and let $U[0]=V[0]:=\Id$.

\smallskip

\emph{Stage $s+1$}. Let $R$ be the highest priority requirement requiring
attention at $s+1$: notice that such an $R$ exists since almost all
requirements are initialized, and thus requiring attention.

We distinguish the various possible cases for $R$, where $T \in \{U,V\}$.

\smallskip
$R=F_e^T$: Set $r_{R,U}$ and $r_{R,V}$ to be the least equivalence strings of
length $\ge e_{R}$, which extend $r^-_{R,U}$ and $r^-_{R,V}$, respectively.

\smallskip
$R=P_e^T$:
\begin{enumerate}
\item If $R$ is initialized then set $r_{R,U}=r^-_{R,U}$ and
    $r_{R,V}=r^-_{R,V}$;

\item otherwise we find suitable $a,b, \sigma$: act as in the description
    of the strategy, picking the least such suitable triple and setting new
    suitable values for $r_{R,U}$ and $r_{R,V}$;
\end{enumerate}

\smallskip
$R=I_e^{U,V}$ or $R=I_e^{V,U}$:
\begin{enumerate}
\item If $R$ is initialized then define a new witness $x_{R}$, and act as
    in the description of the strategy by setting new suitable values for
    $r_{R,U}$ and $r_{R,V}$;

\item otherwise $\phi_e(x_{R})$ has already converged, so act as in the
    description of the strategy and setting new suitable values for
    $r_{R,U}$ and $r_{R,V}$.

\end{enumerate}

Let $U[s+1]:=T_{r_{R,U}}$, $V[s+1]:=T_{r_{R,V}}$. Initialize all strategies
of lower priority than $R$ and go to next stage.

At the end of the construction, $U$ and $V$ are the relations formed by the
pairs which appear at co-finitely many stages in the approximations  $U[s]$
and $V[s]$.

\subsubsection*{The verification}
We show by induction on the priority ordering that each requirement requires
attention (and thus acts) only finitely often and it is eventually satisfied.

Suppose that the claim is true of every requirement $R'<R$, and let $t$ be
the least stage after which no $R'<R$ requires attention. It is now immediate
to see that  after $t$, $R$ requires attention only finitely many times. It
requires attention a first time at $t$ because it is still initialized, and
then: if $R$ is an $F$-requirement then it will never require attention
again; if $R$ is a $P$-requirement then it may require attention at most once
more if later finds suitable $a,b$ and $\sigma$; finally if $R$ is an
$I$-requirement trying to achieve that $\phi_{e}$ is not a reduction, then it
may require attention at most once more again if $\phi_{e}(x_{R})$ converges.
This shows that there is a least stage $t_{R}$ after which neither any $R'<R$ nor
$R$ requires attention.

Next, we show that for every $R$, with priority position $e_{R}$, the number
of equivalence classes $[i]_{T}$, for $i < e_{R}$ and $T\in \{U,V\}$, does
not change after $t_{R}$, nor does any such equivalence class, except for
$[0]_{T}$, changes after this stage. This is so because no requirement of
priority position $j \ge e_{R}$ can add numbers to $[i]_{T}$ if $i < e_{R}$
and $i\cancel{\rel{T}} 0$,  due to the restraints imposed by the
$F$-requirements of priority position $<e_{R}$: these restraints are
respected by the lower priority $P$-requirements as after $t_{R}$ they may
collapse only numbers currently not in these equivalence classes (nor in
$[0]_{T}$ as well), and by the $I$-requirements as by initialization they
pick their witnesses not in the current approximations to these equivalence
classes. This claim shows also that the $F$-requirements are satisfied.

In the other cases, the last action performed by $R$ at the last stage at
which it requires attention achieves satisfaction of $R$, as is clear by the
description of the strategies given before the construction. We verify this
in more detail only for the $P$-requirements, leaving to the reader the other
cases. Suppose $R=P^{U}_{e}$, the case $R=P^{V}_{e}$ being similar. Assume
that $W_{e}^{V}$ is an infinite transversal for $U$. Then since the number of
equivalence classes $[i]_{T}$, for $i < e_{R}$ and $T\in \{U,V\}$, does not
change after $t_{R}$ and is finite, there is a stage $t \ge t_{R}$ at which
$W_{e}^{\sigma}$, with $\sigma \subset c_{V}$, enumerates two distinct
numbers $a,b$ not in any of these equivalence classes. But any finite string
which is an initial segment of $c_{V}$ is an equivalence string, so we can
certainly find at $t$ an equivalence string $\sigma \supseteq r^-_{R,V}$ with
$a, b \in W_{e}^{\sigma}$ at $t$, so that $R$'s action (if $R$ is still
active) guarantees that $a \rel{U} b$ and $a, b \in W_{e}^V$, contradicting
that $W_{e}^{V}$ is a transversal.

We finally show that each $T \in \{U, V\}$ is $\omega$-c.e.\,. Let $u \in
\omega$, and let $R$ be the least $F$-requirement (say $R= F^T_e$) with
$e_{R}> u$. The last time $R$ requires attention it sets a restraint $r_{R,
T}$ of length $> u$ which hereinafter will be an initial segment of the
approximation to the final characteristic function of $T$, and thus the value
$c_T(u)$ will never change again. This shows that the number of possible
changes of $c_T(u)$ is bound by the number of times $R$ is initialized, which
is bound by the number of times requirements $R' < R$ act. Since such a
requirement $R' < R$ can act at most twice after its last initialization,
this shows that the number of possible changes of $c_T(u)$ is bound by
$2^{e_{R}}$, where $e_{R}$ can be effectively computed from $u$. In
conclusion $T$ is $\omega$-c.e.\,.
\end{proof}

We conclude the paper with the following open question.

\begin{question}
For which $a \in \mathcal{O}$, there exist equivalence relations properly in
$\Sigma^{-1}_a$ with $\sup$ in $\Sigma^{-1}_a$?
\end{question}

\medskip

\subsection*{Acknowledgements}
Part of the research contained in this paper was carried out while Bazhenov,
Yamaleev, and San Mauro were visiting the Department of Mathematics of
Nazarbayev University, Astana. These authors wish to thank Nazarbayev
University for its hospitality.

\bibliographystyle{plain}


\end{document}